\documentclass[12pt,a4paper]{amsart}
\usepackage{a4wide}
\usepackage{amsmath, amssymb, amsfonts,enumerate}
\usepackage[all]{xy}
\usepackage{amscd}
\usepackage{comment}
\usepackage{mathtools}

\newtheorem{theorem}{Theorem}[section]
\newtheorem{lemma}[theorem]{Lemma}
\newtheorem{corollary}[theorem]{Corollary}
\newtheorem{proposition}[theorem]{Proposition}

 \theoremstyle{definition}
 \newtheorem{definition}[theorem]{Definition}
 \newtheorem{remark}[theorem]{Remark}

 \newtheorem{example}[theorem]{Example}
\newtheorem{examples}[theorem]{Examples}

\numberwithin{equation}{section}
\newcommand {\N}{\mathbb{N}} %% positive integers
\newcommand {\Z}{\mathbb{Z}} %% integers
\newcommand {\R}{\mathbb{R}} %% reals
\newcommand {\Q}{\mathbb{Q}} %% rationals
\newcommand {\C}{\mathbb{C}} %% complex

\newcommand{\T}{\mathbb{T}}

\newcommand{\CC}{\mathcal{C}}

   \DeclareMathOperator{\GL}{GL}

\DeclareMathOperator{\Fix}{Fix}

\DeclareMathOperator{\Ker}{Ker}

\DeclareMathOperator{\M}{Mat}

\DeclareMathOperator{\supp}{supp}

\begin{document}
\title
[Garden of Eden theorem for harmonic models]{Homoclinically expansive actions and a Garden of Eden theorem for harmonic models}
\author{Tullio Ceccherini-Silberstein}
\address{Dipartimento di Ingegneria, Universit\`a del Sannio, 82100 Benevento, Italy}
\email{tullio.cs@sbai.uniroma1.it}
\author{Michel Coornaert}
\address{Universit\'e de Strasbourg, CNRS, IRMA UMR 7501, F-67000 Strasbourg, France}
\email{michel.coornaert@math.unistra.fr}
\author{Hanfeng Li}
\address{Department of Mathematics, Chongqing University, Chongqing 401331, China}
\address{Department of Mathematics, SUNY at Buffalo, Buffalo, NY 14260-2900, USA}
\email{hfli@math.buffalo.edu}
\begin{abstract}
Let $\Gamma$ be a countable Abelian group and $f \in \Z[\Gamma]$, where $\Z[\Gamma]$ denotes the integral group ring of $\Gamma$.
Consider the Pontryagin dual $X_f$ of the cyclic $\Z[\Gamma]$-module $\Z[\Gamma]/\Z[\Gamma] f$ and
suppose that $f$ is weakly expansive (e.g., $f$ is invertible in $\ell^1(\Gamma)$, or, when $\Gamma$ is not virtually $\Z$ or $\Z^2$,
$f$ is well-balanced) and that $X_f$ is connected.
We prove that if $\tau \colon X_f \to X_f$ is a $\Gamma$-equivariant continuous map, then $\tau$ is surjective if and only if the restriction of $\tau$ to each $\Gamma$-homoclinicity class is injective.
We also show that this equivalence remains valid in the case when
$\Gamma = \Z^d$ and $f \in \Z[\Gamma] = \Z[u_1,u_1^{-1}, \ldots, u_d, u_d^{-1}]$ is an irreducible atoral polynomial
such that its zero-set $Z(f)$ is contained in the image of the intersection of $[0,1]^d$ and a finite union of hyperplanes in
$\R^d$ under the quotient map $\R^d \to \T^d$ (e.g., when $d \geq 2$ such that $Z(f)$ is finite).
These two results are analogues of the classical Garden of Eden theorem of Moore and Myhill for cellular automata with finite alphabet over
$\Gamma$.
\end{abstract}
\date{\today}
\maketitle

\tableofcontents
% SECTION 1
\section{Introduction}
Consider a dynamical system $(X, \alpha)$, consisting of a compact metrizable space $X$,
called the \emph{phase space},
equipped with a continuous  action  $\alpha$
of a countable group $\Gamma$.
Let  $d$ be a metric on $X$ that is compatible with the topology.
Two points $x, y \in X$ are said to be \emph{homoclinic} if
$\lim_{\gamma \to \infty} d(\gamma x, \gamma y) = 0$, i.e.,
for every $\varepsilon > 0$, there exists a finite set $F \subset \Gamma$ such that
$d(\gamma x,\gamma y) <  \varepsilon$
for all $\gamma \in \Gamma \setminus F$.
Homoclinicity is an equivalence relation on $X$.
By compactness of $X$, this relation does not depend on the choice of the compatible metric $d$.
 A map with source set   $X$ is called \emph{pre-injective}  (with respect to $\alpha$)
 if its restriction to each homoclinicity class is injective.
  \par
An \emph{endomorphism} of the dynamical system $(X,\alpha)$ is a continuous map
 $\tau \colon X \to X$ that is $\Gamma$-equivariant
 (i.e., $\tau(\gamma x) = \gamma \tau(x)$ for all $\gamma \in \Gamma$ and $x \in X$).
 \par
The original Garden of Eden theorem  is a statement in symbolic dynamics that characterizes surjective endomorphisms of shift systems with finite alphabet.
To be more specific, let us fix a compact metrizable space  $A$, called the \emph{alphabet}.
Given a countable group $\Gamma$,
the \emph{shift} over the group $\Gamma$ with alphabet $A$ is the dynamical system
$(A^\Gamma,\sigma)$, where
$A^\Gamma = \{x \colon \Gamma \to A\}$ is equipped with the product topology and
  $\sigma $ is the action defined by
  $\gamma x(\gamma') \coloneqq  x(\gamma^{-1} \gamma')$ for all $x \in A^\Gamma$ and $\gamma,\gamma' \in \Gamma$.
The \emph{Garden of Eden theorem}  states that, under the hypotheses that the group  $\Gamma$ is amenable and the alphabet $A$ is finite,    an endomorphism of $(A^\Gamma,\sigma)$ is surjective if and only if it is pre-injective.
It was first proved for  $\Gamma = \Z^d$ by Moore and Myhill in the early 1960s.
Actually, the implication
surjective $\implies$ pre-injective  was first proved by Moore in~\cite{moore}
while the converse implication was established shortly after by
Myhill in~\cite{myhill}.
The Garden of Eden theorem was subsequently extended to
finitely generated groups of subexponential growth by Mach{\`{\i}} and Mignosi~\cite{machi-mignosi}
and finally to all countable amenable groups by
Mach{\`{\i}},  Scarabotti, and the first author in~\cite{ceccherini}.
 \par
Let us  say that the  dynamical system $(X,\alpha)$ has the \emph{Moore property} if every surjective endomorphism of $(X,\alpha)$ is pre-injective and that it has the \emph{Myhill property} if
every pre-injective endomorphism of $(X,\alpha)$ is surjective.
We say that the dynamical system $(X,\alpha)$ has the \emph{Moore-Myhill property},
or that it satisfies the \emph{Garden of Eden theorem},
if it has both the Moore and the Myhill properties.

The goal of the present paper is to establish a version of the Garden of Eden theorem for principal algebraic dynamical systems associated with weakly expansive polynomials over countable Abelian groups and with connected phase space.
By an \emph{algebraic dynamical system}, we mean a dynamical system of the form
$(X,\alpha)$, where $X$ is a compact metrizable Abelian group and $\alpha$ is an  action of a countable group $\Gamma$
on $X$ by continuous group automorphisms. Note that, in this case, the set $\Delta(X,\alpha) = \{x \in X: x$ is homoclinic to $0_X\} \subset X$, where $0_X$ is the identity element of $X$, is a subgroup of $X$, called the \emph{homoclinic group} and two points $x,y \in X$
are homoclinic if and only if $x-y \in \Delta(X,\alpha)$, that is, they belong to the same coset of $\Delta(X,\alpha)$ in $X$.
By Pontryagin duality, algebraic dynamical systems with acting group $\Gamma$ are in one-to-one correspondence with countable left
$\Z[\Gamma]$-modules.
Here $\Z[\Gamma]$ denotes the integral group ring of $\Gamma$.
This correspondence has been intensively studied in the last decades
and revealed fascinating connections  between commutative algebra, number theory, harmonic analysis, ergodic theory,   and dynamical systems (see in particular the monograph~\cite{schmidt-book} and the survey~\cite{lind-schmidt-survey-heisenberg}).
\par
Let $f \in \Z[\Gamma]$ and consider the cyclic left $\Z[\Gamma]$-module
$M_f \coloneqq \Z[\Gamma]/  \Z[\Gamma] f$ obtained by quotienting the ring $\Z[\Gamma]$ by the principal left ideal generated by $f$.
The algebraic dynamical system  associated by Pontryagin duality with $M_f$
is denoted by $(X_f,\alpha_f)$ and is  called the \emph{principal algebraic dynamical system} associated with $f$.
\par
We denote by $\CC_0(\Gamma)$ the real vector space of all functions $g \colon \Gamma \to \R$ vanishing at infinity (i.e., for every $\varepsilon > 0$ there exists a finite subset $\Omega \subset \Gamma$ such that $|g_\gamma| \leq \varepsilon$ for all $\gamma \in \Gamma \setminus \Omega$).
Moreover, for $f \in \Z[\Gamma]$ and $g \in \CC_0(\Gamma)$ we denote by $fg \in \CC_0(\Gamma)$ their
convolution product (see Subsection \ref{ss:convolution}).

\begin{definition}
\label{def:we}
A polynomial $f \in \Z[\Gamma]$ is said to be {\it weakly expansive} provided:
% the following hold:
\begin{enumerate}[{\rm (we-1)}]
\item $\forall g \in \CC_0(\Gamma)$, $fg = 0$ $\Rightarrow$ $g = 0$;
\item $\exists \omega \in \CC_0(\Gamma)$ such that $f\omega = 1_\Gamma$.
\end{enumerate}
\end{definition}
\par
Our first result is the following.

\begin{theorem}[Garden of Eden theorem for algebraic actions associated with weakly expansive polynomials]
\label{t:main-result}
Let $\Gamma$ be a countable Abelian group and $f \in \Z[\Gamma]$.
Suppose that $f$ is weakly expansive and that $X_f$ is connected.
Then the dynamical system $(X_f,\alpha_f)$ has the Moore-Myhill property.
\end{theorem}

There are two main ingredients in our proof of Theorem \ref{t:main-result}.
The first one, Corollary \ref{c:endo-princ}, is a rigidity result (a generalization of \cite[Corollary 1]{bhattacharya})
for algebraic dynamical systems associated with weakly expansive polynomials and with connected phase space.
We use it to prove that, under the above conditions, every endomorphism of $(X_f,\alpha_f)$ is affine with linear part of the form
$x \mapsto r x$ for some $r \in \Z[\Gamma]$.
The second one, Theorem \ref{t:we-implies}, a generalization of \cite[Lemma 4.5]{lind-schmidt}),
asserts that, if $(X_f,\alpha_f)$ is weakly expansive, then its homoclinic group $\Delta(X_f,\alpha_f)$, equipped with
%the discrete topology and
the induced action of $\Gamma$, is dense in $X_f$ and isomorphic, as a $\Z[\Gamma]$-module, to $\Z[\Gamma]/ \Z[\Gamma] f^*$, where
$f^* \in \Z[\Gamma]$ is defined by $(f^*)_\gamma \coloneqq f_{\gamma^{-1}}$ for all $\gamma \in \Gamma$.

Recall that a  dynamical system  $(X,\alpha)$ is called \emph{expansive} if there exists a constant
$\varepsilon_0  > 0$ such that, for every pair of distinct points $x,y \in X$, there exists an element
$\gamma \in \Gamma$ such that $d(\gamma  x,\gamma  y) >  \varepsilon_0$.
Such a constant $\varepsilon_0$ is called an \emph{expansivity constant} for $(X,\alpha,d)$.
The fact that $(X,\alpha)$ is expansive or not does not depend on the choice of the metric $d$.
For instance, the shift system $(A^\Gamma,\sigma)$ is expansive for every countable group $\Gamma$ whenever the alphabet $A$ is finite.
\par
Let $f \in \Z[\Gamma]$ and suppose that the associated  principal algebraic dynamical system $(X_f,\alpha_f)$ is expansive. In Corollary
\ref{c:espansive-vs-weakly expansive} we show that $f$ is weakly expansive, and from Theorem \ref{t:main-result} we thus deduce the following:

\begin{corollary}
\label{c:expansive}
Let $\Gamma$ be a countable Abelian group and $f \in \Z[\Gamma]$.
Suppose that the principal algebraic dynamical system $(X_f,\alpha_f)$ associated with $f$ is expansive and that $X_f$ is connected.
Then the  dynamical system $(X_f,\alpha_f)$ has the Moore-Myhill property. \qed
\end{corollary}

This result was obtained by the first two named authors in \cite{ccs-arXiv} and, shortly after, as a particular case of a
much more general Garden of Eden theorem for expansive actions (where $\Gamma$ is amenable and no connectedness of the phase
space is assumed) proved by the third author in \cite{Li-17}.

Recall that a polynomial $f \in \R[\Gamma]$ is said to be \emph{well-balanced} (cf.\ \cite[Definition 1.2]{Bowen-Li}) if:
\begin{enumerate}[{\rm (wb-1)}]
\item $\sum_{\gamma \in \Gamma} f_\gamma = 0$,
\item $f_\gamma \leq 0$ for all $\gamma \in \Gamma \setminus \{1_\Gamma\}$,
\item $f_\gamma = f_{\gamma^{-1}}$ for all $\gamma \in \Gamma$ (i.e., $f$ is \emph{self-adjoint}),
\item and $\supp(f) \coloneqq \{\gamma \in \Gamma: f_\gamma \neq 0\}$, the \emph{support} of $f$, generates $\Gamma$.
\end{enumerate}

If $f\in \Z[\Gamma]$ is well-balanced, the associated dynamical system $(X_f, \Gamma)$ is called a \emph{harmonic model}.
For $\Gamma = \Z^d$ and $f = 2d - \sum_{i=1}^d (u_i + u_i^{-1}) \in \Z[u_1, u_1^{-1}, \ldots, u_d, u_d^{-1}] = \Z[\Z^d]$, the
corresponding harmonic model shares interesting measure theoretic and entropic properties with other different models in mathematical
physics, probability theory, and dynamical systems such as the Abelian sandpile model, spanning trees,
and the dimer models \cite{schmidt-verbitskiy, Bowen-Li}.
Since a well-balanced polynomial $f \in \Z[\Gamma]$, with $\Gamma$ infinite countable not virtually $\Z$ or $\Z^2$,
is weakly expansive (cf.\ Proposition \ref{p:wb-implies-we}), from Theorem \ref{t:main-result} we deduce:

\begin{corollary}[Garden of Eden theorem for harmonic models]
\label{c:harmomic}
Let $\Gamma$ be an infinite countable Abelian group which is not virtually $\Z$ or $\Z^2$  (e.g.~$\Gamma = \Z^d$, with $d \geq 3$).
Suppose that $f \in \Z[\Gamma]$ is well-balanced and that $X_f$ is connected.
Then the  dynamical system $(X_f,\alpha_f)$ has the Moore-Myhill property. \qed
\end{corollary}

Let $(X,\alpha)$ be an algebraic dynamical system and $1 \leq p \leq \infty$.
In \cite{Chung-Li} (see also Subsection \ref{s:p-exp-and-p-homoclinic}), the notion of $p$-homoclinic group
(denoted $\Delta^p(X,\alpha) \subset X$) associated with $(X,\alpha)$, was introduced.
We then say that a map $\tau \colon X \to X$ is $p$-pre-injective if the restriction of $\tau$ to each
coset of the $p$-homoclinic group $\Delta^p(X,\alpha)$ is injective.
Note that $\Delta^\infty(X,\alpha) = \Delta(X,\alpha)$ so that $\infty$-pre-injectivity is the same thing as pre-injectivity.
\par
If $\Gamma = \Z^d$, then any polynomial $f \in \R[\Gamma]$ may be regarded, by duality, as a function
on $\widehat{\Gamma} = \T^d$.
We denote by $Z(f) \coloneqq \{(t_1, t_2, \ldots, t_d) \in \T^d: f(t_1, t_2, \ldots, t_d) = 0\}$ its zero-set.
Recall that an irreducible polynomial $f$ is \emph{atoral} \cite[Definition 2.1]{lind-schmidt-verbitskiy-2} if there is some
$g\in \Z[\Gamma]$ such that $g\not \in \Z[\Gamma] f$ and $Z(f)\subset Z(g)$.
\par
We are now in position to state the following:

\begin{theorem}[A Garden of Eden theorem for irreducible atoral polynomials]
\label{t:GOE-irr-finite-zero-set}
Let $f\in \Z[\Z^d]$ be an irreducible atoral polynomial such that $Z(f)$ is contained in the image of the intersection of $[0,1]^d$ and a finite union of
hyperplanes in $\R^d$ under the natural quotient map $\R^d \to \T^d$ (e.g., when $d \geq 2$ such that $Z(f)$ is finite).
Let $\tau \colon X_f \to X_f$ be a $\Gamma$-equivariant continuous map. Then the following conditions are equivalent:
\begin{enumerate}[{\rm (a)}]
\item $\tau$ is surjective,
\item $\tau$ is pre-injective,
\item $\tau$ is $p$-pre-injective for all $1 \leq p \leq \infty$,
\item $\tau$ is $1$-pre-injective.
\end{enumerate}
In particular, $(X_f, \alpha_f)$ satisfies the Moore-Myhill property.
\end{theorem}

Our motivation for the present work originated from a sentence of Gromov \cite[p.~195]{gromov-esav}
suggesting that the Garden of Eden theorem could be extended to dynamical systems with a suitable hyperbolic flavor other than shifts and subshifts.
A first step in that direction was made in~\cite{csc-anosov-tori}, where the two first named authors proved that all Anosov diffeomorphisms on tori generate $\Z$-actions with the Moore-Myhill property, and another one in~\cite{csc-ijm-2015},
where sufficient conditions for expansive actions of countable amenable groups to have the Myhill property were presented. Finally,
in \cite{Li-17} the third named author proved the very general Garden of Eden theorem for expansive actions of amenable groups we alluded to above.
\par
The paper is organized as follows.
Section~\ref{sec:background} introduces notation and collects background material on Pontryagin duality and algebraic dynamical systems.
In particular, in Subsection \ref{sec:homoclinic} we present a characterization of the homoclinic group.
In Section \ref{s:weak-forms} we consider several weak forms of expansivity: first we recall from \cite{Chung-Li} the notions
of $p$-expansivity ($1 \leq p \leq \infty$) and $p$-homoclinicity for general algebraic actions. Then we
introduce and study the notion of homoclinically expansive action: when the algebraic dynamical system is finitely generated, homoclinic expansivity is a stronger condition than $p$-expansivity for $1 \leq p < \infty$ (Proposition \ref{P-h expansive basic}), and a characterization is derived (Theorem \ref{t:char-homo-exp}). Then, we study principal algebraic actions associated with weakly expansive polynomials with an emphasis on the corresponding homoclinic group. In Subsection \ref{ss:expansivity} we then show that polynomials yielding principal algebraic expansive actions are weakly expansive and, in Subsection \ref{sec:harmonic}, we prove that well-balanced polynomials, with not virtually $\Z$ or $\Z^2$ infinite countable group, are weakly expansive as well.
In Section~\ref{sec:affine}, we discuss topological rigidity of algebraic dynamical systems associated with weakly expansive polynomials.
The proof of Theorem~\ref{t:main-result} is then given in Section~\ref{sec:proof}.
In the following section we discuss the notion of atorality for irreducible polynomials in $\Z[\Z^d]$, we present a few examples, and give the proof of Theorem \ref{t:GOE-irr-finite-zero-set}.
In the last section, we collect some final remarks. In particular, we exhibit some examples showing that Theorem~\ref{t:main-result} becomes false if weak expansivity of $f \in \Z[\Gamma]$ is replaced by the weaker hypothesis that the homoclinic group $\Delta(X_f,\alpha_f)$ is dense in $X_f$, or that the dynamical system $(X_f,\alpha_f)$ is mixing. We also introduce and discuss the notions of $p$-pre-injectivity,
$p$-Moore, and $p$-Myhill properties for algebraic actions, and prove some variations on the Garden of Eden theorem in this framework.\\

\noindent
{\bf Acknowledgments}. Hanfeng Li was partially supported by NSF and NSFC grants.

% SECTION 2
\section{Background material and preliminaries}
\label{sec:background}

\subsection{Group actions}
Let $\Gamma$ be a countable group.
We use multiplicative notation for the group operation in $\Gamma$ and denote by $1_\Gamma$ its identity element.
\par
An \emph{action} of  $\Gamma$ on a set $X$
is a map $\alpha \colon \Gamma \times X \to X$ such that
$\alpha(1_\Gamma,x) = x$ and
$\alpha(\gamma_1,\alpha(\gamma_2,x)) = \alpha(\gamma_1 \gamma_2,x)$ for all $\gamma_1,\gamma_2 \in \Gamma$ and $x \in X$.
In the sequel, to simplify, we shall write
$\gamma x$ instead of $\alpha(\gamma, x)$,  if there is no risk of confusion.
\par
If $\alpha$ is an action of $\Gamma$ on a set $X$, we denote by $\Fix(X,\alpha)$ the set of points of $X$ that are \emph{fixed} by $\alpha$, i.e., the set of points $x \in X$ such that $\gamma x = x$ for all
$\gamma \in \Gamma$.
\par
If $\Gamma$ acts on two sets $X$ and $Y$, a map
$\tau \colon X \to Y$ is said to be $\Gamma$-\emph{equivariant} if one has
$\tau(\gamma x) = \gamma \tau(x)$ for all $\gamma \in \Gamma$ and $x \in X$.

\subsection{Convolution}
\label{ss:convolution}
Let $\Gamma$ be a countable group.
We denote by $\ell^\infty(\Gamma)$  the vector space consisting of all formal series
$$
f = \sum_{\gamma \in \Gamma} f_\gamma \gamma,
$$
with coefficients $f_\gamma \in \R$ for all $\gamma \in \Gamma$ and
$$
\Vert f \Vert_\infty \coloneqq \sup_{\gamma \in \Gamma} |f_\gamma| < \infty.
$$
For $1 \leq p < \infty$ we denote by $\ell^p(\Gamma)$ the vector subspace of $\ell^\infty(\Gamma)$
consisting of all $f \in \ell^\infty(\Gamma)$ such that
$$
\Vert f \Vert_p \coloneqq \left(\sum_{\gamma \in \Gamma} |f_\gamma|^p\right)^{\frac{1}{p}} < \infty.
$$
Note that $\ell^1(\Gamma) \subset \ell^p(\Gamma) \subset \ell^q(\Gamma) \subset \ell^\infty(\Gamma)$ for all $1 < p < q < \infty$.
When $f \in \ell^\infty(\Gamma)$ and $g \in \ell^1(\Gamma)$ (resp.\ $f \in \ell^1(\Gamma)$ and $g \in \ell^\infty(\Gamma)$)
we define the \emph{convolution product} $f g \in \ell^\infty(\Gamma)$ by setting
\begin{equation}
\label{e:convolution}
(f g)_\gamma \coloneqq
\sum_{\substack{\gamma_1, \gamma_2 \in \Gamma:\\ \gamma_1\gamma_2 = \gamma}}  f_{\gamma_1} g_{\gamma_2}
= \sum_{\delta \in \Gamma}  f_{\gamma \delta^{-1}}  g_\delta
\end{equation}
for all $\gamma \in \Gamma$.
Note that $\|f g \|_\infty \leq \|f\|_\infty \cdot \|g\|_1$ (resp.\ $\|f g \|_\infty \leq \|f\|_1 \cdot \|g\|_\infty$).
We have the associativity rule
\begin{equation}
\label{e:associativity}
(f g) h = f (g h)  \mbox{ for all $f \in \ell^\infty(\Gamma)$, $g,h \in \ell^1(\Gamma)$ (resp.\ $f,g \in \ell^1(\Gamma)$, $h \in \ell^\infty(\Gamma)$).}
\end{equation}

We denote by $\R[\Gamma] = \{f \in \ell^\infty(\Gamma): f_\gamma = 0$ for all but finitely many $\gamma \in \Gamma\}$ and by $\Z[\Gamma] = \{f \in \R[\Gamma]: f_\gamma \in \Z$ for all $\gamma \in \Gamma\}$ the \emph{real} and, respectively, the \emph{integral group ring} of
$\Gamma$. Observe that the convolution product extends the group operation on $\Gamma \subset \Z[\Gamma] \subset \R[\Gamma]$.
\par
Note also that, as a $\Z$-module, $\Z[\Gamma]$ is free with base $\Gamma$.
\par
If we take $\Gamma = \Z^d$, then $\Z[\Gamma]$ is the Laurent polynomial ring
$R_d \coloneqq \Z[u_1^{\pm 1}, \dots, u_d^{\pm 1}]$ on $d$ commuting indeterminates
$u_1,\dots, u_d$.
\par

Recall that $\CC_0(\Gamma)$ denotes the vector space consisting of all functions $f \colon \Gamma \to
\R$ vanishing at infinity: we express this condition by writing $\lim_{\gamma \to \infty} f(\gamma) = 0$.
We then have the inclusions
\begin{equation}
\label{e:inclusion}
\Gamma \subset \Z[\Gamma] \subset \R[\Gamma] \subset \ell^1(\Gamma) \subset \ell^p(\Gamma) \subset \CC_0(\Gamma) \subset \ell^\infty(\Gamma),
\end{equation}
for all $1 \leq p < \infty$.
Moreover, there is a natural involution $f \mapsto f^*$ on $\ell^\infty(\Gamma)$ defined by
\begin{equation}
\label{e:involution}
(f^*)_\gamma \coloneqq  f_{\gamma^{-1}}
\end{equation}
for all $f \in \ell^\infty(\Gamma)$ and $\gamma \in \Gamma$. Observe that every set in \eqref{e:inclusion} is $*$-invariant and that
\begin{equation}
\label{e:invoution-product}
(fg)^* = g^*f^* \mbox{ for all $f \in \ell^\infty(\Gamma)$ and $g \in \ell^1(\Gamma)$ (resp.\ $f \in \ell^1(\Gamma)$ and $g \in \ell^\infty(\Gamma)$).}
\end{equation}
The normed space $(\ell^1(\Gamma),\|\cdot\|_1)$ is a unital Banach *-algebra for the convolution product
and the involution. The unity element of $\ell^1(\Gamma)$ is $1_\Gamma$.

\begin{lemma}
\label{l:associative}
Let $f,h \in \R[\Gamma]$ and $g \in \ell^\infty(\Gamma)$. Then $(fg)h = f(gh)$.
\end{lemma}
\begin{proof}
By linearity we can reduce to the case when $f = \gamma$ and $h = \delta$ with $\gamma, \delta \in \Gamma$.
For $\eta \in \Gamma$ we have
\[
((\gamma g)\delta)_\eta = (\gamma g)_{\eta \delta^{-1}} = g_{\gamma^{-1}(\eta \delta^{-1})} =_* g_{(\gamma^{-1}\eta) \delta^{-1}} = (g \delta)_{\gamma^{-1}\eta} = (\gamma (g\delta))_\eta,
\]
where $=_*$ follows from the associative rule in $\Gamma$.
\end{proof}

Let now $k,n \in \N$ and denote by $\M_{n,k}(\Z[\Gamma]) \coloneqq \{(a^{ij})_{\substack{1 \leq i \leq n\\ 1 \leq j \leq k}}: a^{ij} \in \Z[\Gamma]\}$ the space of all $n$-by-$k$ matrices with coefficients in the group ring $\Z[\Gamma]$. We identify
$\Z[\Gamma]^k$ (resp.\ $\Z[\Gamma]^n$) and $\M_{1,k}(\Z[\Gamma])$ (resp.\ $\M_{1,n}(\Z[\Gamma])$) so that if $g = (g^1,g^2, \ldots, g^n) \in
\Z[\Gamma]^n$ and $A = (a^{ij})_{\substack{1 \leq i \leq n\\ 1 \leq j \leq k}} \in \M_{n,k}(\Z[\Gamma])$, the element
$gA \in \Z[\Gamma]^k$ is defined by $(gA)^j \coloneqq \sum_{i=1}^n g^i a^{ij} \in \Z[\Gamma]$, that is,
\begin{equation}
\label{e:def-A-convolution}
(gA)^j_\gamma = \sum_{i=1}^n \sum_{\eta \in \Gamma} g^i_{\gamma \eta} a^{ij}_{\eta^{-1}}
\end{equation}
for all $j=1,2, \ldots,k$ and $\gamma \in \Gamma$.
Given $A = (a^{ij})_{\substack{1 \leq i \leq n\\ 1 \leq j \leq k}} \in \M_{n,k}(\Z[\Gamma])$, we define
\begin{equation}
\label{e:def-A-star}
A^* = \left((a^{ji})^*\right)_{\substack{1 \leq i \leq k\\ 1 \leq j \leq n}} \in \M_{k,n}(\Z[\Gamma]).
\end{equation}
Note that when $k=n=1$ and $A = f \in \Z[\Gamma]$, then \eqref{e:def-A-star} reduces to \eqref{e:involution}.
Also, we set
\[
\|A\|_\infty = \sup_{\substack{1 \leq i \leq n\\ 1 \leq j \leq k}}
\sup_{\gamma \in \Gamma} |a^{ij}_\gamma|
\]
and
\[
\|A\|_1 = \sum_{\substack{1 \leq i \leq n\\ 1 \leq j \leq k}}
\|a^{ij}\|_1 = \sum_{\substack{1 \leq i \leq n\\ 1 \leq j \leq k}} \sum_{\gamma \in \Gamma} |a^{ij}_\gamma|.
\]

\subsection{Pontryagin duality}
\label{subsec:pontryagin}
 Let us briefly review some basic facts and results regarding Pontryagin duality.
For more details and complete proofs,
the reader is refered to~\cite{morris}.
\par
Let $X$ be an LCA group, i.e., a locally compact, Hausdorff, Abelian topological group.
A continuous group morphism from $X$ into the circle $\T \coloneqq \R/\Z$
is called a \emph{character} of $X$.
The set $\widehat{X}$ of all characters of $X$,
equipped with pointwise multiplication and the topology of uniform convergence on compact sets,
is also an LCA  group,
  called  the \emph{character group} or  \emph{Pontryagin dual}  of $X$.
  \par
The natural map $\langle \cdot, \cdot  \rangle \colon \widehat{X} \times X \to \T$,
given by $\langle \chi, x \rangle = \chi(x)$ for all $x \in X$ and $\chi \in \widehat{X}$ is bilinear and non-degenerate.
Moreover,
the evaluation map $\iota \colon X \to \widehat{\widehat{X}}$,
defined by $\iota(x)(\chi) \coloneqq \langle\chi,x\rangle$,
is a topological group    isomorphism  from $X$ onto its bidual
$\widehat{\widehat{ X}}$.
This canonical isomorphism is used to identify $X$ with $\widehat{\widehat{ X}}$.
\par
The space $X$ is
compact (resp.~discrete, resp.~metrizable, resp.~$\sigma$-compact)
if and only if $\widehat{X}$ is
discrete (resp.~compact, resp.~$\sigma$-compact, resp.~metrizable).
In particular, $X$ is compact and metrizable if and only if $\widehat{X}$ is discrete and countable.
When $X$ is compact, $X$ is connected if and only if $\widehat{X}$ is a \emph{torsion-free} group  (i.e., a group with  no non-trivial elements of finite order).
\par
If $X$ is an LCA group and $Y$ a closed subgroup of $X$, then $X/Y$ is an LCA group
whose Pontryagin dual is canonically isomorphic, as a topological group,
to the closed subgroup $Y^\perp$ of $\widehat{X}$ defined by
\[
Y^\perp \coloneqq \{\chi \in \widehat{X} : \langle \chi, y \rangle = 0 \text{ for all } y \in Y\}.
\]
\par
Let $X, Y$ be LCA groups and $\varphi \colon X \to Y$ a continuous group morphism.
The map $\widehat{\varphi} \colon \widehat{Y} \to \widehat{X}$, defined by
$\widehat{\varphi}(\chi) \coloneqq  \chi \circ \varphi$ for all $\chi \in \widehat{Y}$ is a continuous group morphism, called the \emph{dual} of $\varphi$.
If we identify $X$ and $Y$ with their respective biduals,
then $\widehat{\widehat{\varphi}} =\varphi$.
If $\varphi$ is surjective, then $\widehat{\varphi}$ is injective. Conversely, if $\varphi$ is injective, then $\widehat{\varphi}$ has
dense image (and it may happen that $\widehat{\varphi}$ is not surjective; however, if $\varphi$ is both injective and open, then $\widehat{\varphi}$ is surjective).
As a consequence, if $X$ and $Y$ are either both compact or both discrete, then  $\varphi$ is injective (resp.~surjective) if and only if $\widehat{\varphi}$ is surjective (resp.~injective)
\cite[Proposition~30]{morris}.
\par
Let $X$ be an LCA group
and suppose that there is a countable group $\Gamma$ acting continuously on $X$ by  group automorphisms. By linearity, this action induces a left $\Z[\Gamma]$-module structure on $X$.
There is a   dual action of $\Gamma$ on $\widehat{X}$ by continuous group automorphisms, defined by
\[
\gamma  \chi(x) \coloneqq \chi(\gamma^{-1}  x) \quad \text{for all } \gamma \in \Gamma, x \in X, \text{ and } \chi \in \widehat{X}.
\]
Therefore  there is also a   left $\Z[\Gamma]$-module structure on $\widehat{X}$.
Note that the canonical topological group isomorphism $\iota \colon X \to \widehat{\widehat{X}}$ is
$\Gamma$-equivariant and hence a left $\Z[\Gamma]$-module isomorphism.

\subsection{Algebraic dynamical systems}
An \emph{algebraic dynamical system}
is a pair $(X,\alpha)$, where $X$ is a compact metrizable Abelian topological group
and $\alpha$ is an action of a countable group $\Gamma$ on $X$ by continuous group automorphisms.
\par
As an example, if $A$ is a compact metrizable Abelian topological group (e.g.~$A = \T$) and
$\Gamma$ a countable group,
then the system $(A^\Gamma,\sigma)$, where
$A^\Gamma = \{x \colon \Gamma \to A\}$ is equipped with the product topology,  and $\sigma$ is the \emph{shift action}, defined by
\[
(\sigma(\gamma,x))(\gamma') \coloneqq x(\gamma^{-1} \gamma') \quad \text{for all  } \gamma, \gamma' \in \Gamma \text{ and }x \in A^\Gamma,
\]
is an algebraic dynamical system.
\par
Let $(X,\alpha)$ be an algebraic dynamical system with acting group $\Gamma$.
As $X$ is compact and metrizable, its Pontryagin dual $\widehat{X}$ is a discrete countable Abelian group.
We have seen at the end of the previous subsection that there is a left $\Z[\Gamma]$-module structure on $\widehat{X}$ induced by the action of $\Gamma$ on $X$.
Conversely, if $M$ is a countable left $\Z[\Gamma]$-module and we equip $M$ with its discrete topology, then its Pontryagin dual $\widehat{M}$ is a compact metrizable Abelian group
and there is, by duality, an action $\alpha_M$ of $\Gamma$ on $\widehat{M}$ by continuous group automorphisms,
so that
$(\widehat{M},\alpha_M)$ is an algebraic dynamical system.
In this way, algebraic dynamical systems with acting group $\Gamma$ are in one-to-one correspondence with countable left $\Z[\Gamma]$-modules.

\subsection{Finitely presented algebraic dynamical systems}
\label{ss:PAAds}
Let $\Gamma$ be a countable group. One says that an algebraic dynamical system $(X,\alpha)$  with acting group $\Gamma$ 
is \emph{finitely generated} provided its Pontryagin dual $\widehat{X}$, equipped with the dual action $\widehat{\alpha}$ of $\Gamma$, 
is a finitely generated left $\Z[\Gamma]$-module.

Recall (cf.\ \cite[Definition 4.25]{Lam}) that for a unital ring $R$, a left $R$-module ${\mathcal M}$ is said to be
\emph{finitely presented} provided there exist $k \in \N$ and some finitely generated left $R$-submodule $J \leq R^k$ such that
${\mathcal M} = R^k/J$.
\par
Let now $k, n \in \N$ and $A \in \M_{n,k}(\Z[\Gamma])$.
Then $\Z[\Gamma]^n A$ is a finitely generated left $\Z[\Gamma]$-submodule of $\Z[\Gamma]^k$.
We denote by $M_A \coloneqq  \Z[\Gamma]^k/ \Z[\Gamma]^n A$ the corresponding
finitely presented left $\Z[\Gamma]$-module.
\par
To simplify notation, let us write $X_A$ instead of $\widehat{M_A}$ and $\alpha_A$ instead
of $\alpha_{M_A}$.
The algebraic dynamical system $(X_A,\alpha_A)$ is called
the \emph{finitely presented algebraic dynamical system} associated with $A$.
\par
For example, if $k=n=1$ and $A = f \in \Z[\Gamma]$, then $M_f = \Z[\Gamma]/\Z[\Gamma]f$,
where $\Z[\Gamma]f$ is the principal left ideal of $\Z[\Gamma]$ generated by $f$, and
the algebraic dynamical system $(X_f,\alpha_f)$ is called
the \emph{principal algebraic dynamical system} associated with $f$.
\par
One  can regard $(X_A,\alpha_A)$ as a \emph{group subshift} of
$((\T^k)^\Gamma,\sigma)$, i.e.,
as a closed  subgroup of $(\T^k)^\Gamma$
that is invariant under the shift action $\sigma$ of $\Gamma$ on $(\T^k)^\Gamma$, in the following way.
The Pontryagin dual of $(\T^k)^\Gamma$ is $\Z[\Gamma]^k$ with pairing
$\langle \cdot,\cdot \rangle \colon \Z[\Gamma]^k \times (\T^k)^\Gamma \to \T$ given by
\[
\langle g, x \rangle = \sum_{j=1}^k\sum_{\eta \in \Gamma} g^j_{\eta} x_j(\eta)
\]
for all $g = (g^1,g^2, \ldots, g^k) \in \Z[\Gamma]^k$  and $x = (x_1,x_2, \ldots, x_k) \in (\T^\Gamma)^k = (\T^k)^\Gamma$.

Therefore, writing $A = (a^{ij})_{\substack{1 \leq i \leq n\\ 1 \leq j \leq k}}$ and $A^i = (a^{i1},a^{i2}, \ldots, a^{ik}) \in \Z[\Gamma]^k$ for $i=1,2, \ldots,n$, we have
\begin{align*}
X_A
& = \widehat{\Z[\Gamma]^k/\Z[\Gamma]^n A} \\
& = (\Z[\Gamma]^n A)^\perp \\
& = \{x \in (\T^k)^\Gamma : \langle g, x\rangle = 0 \text{ for all } g \in \Z[\Gamma]^n A\} \\
& =_{*} \{x \in (\T^k)^\Gamma : \langle \gamma A^i, x\rangle = 0 \text{ for all } \gamma \in \Gamma  \text{ and } i=1,2,\ldots,n\}\\
& = \{x \in (\T^k)^\Gamma : \sum_{j=1}^k\sum_{\eta \in \Gamma} a^{ij}_{\eta} x_j(\gamma \eta)  = 0 \text{ for all } \gamma \in \Gamma
\text{ and } i=1,2,\ldots,n\}\\
& = \{x \in (\T^k)^\Gamma : \sum_{j=1}^k\sum_{\eta \in \Gamma}  x_j(\gamma \eta)(a^{ij})^*_{\eta^{-1}}  = 0 \text{ for all }
\gamma \in \Gamma \text{ and } i=1,2,\ldots,n\},
\end{align*}
where $=_{*}$ follows by taking $g \coloneqq \gamma(0, \ldots, 0, {1_\Gamma},0, \ldots,0)A = \gamma A^i \in \Z[\Gamma]^k$.
In other words,
\begin{equation}
\label{e:X-A-subshift}
X_A= \{x \in (\T^k)^\Gamma : xA^* = 0_{(\T^n)^\Gamma}\},
\end{equation}
with the action $\alpha_A$ of $\Gamma$ on $X_A \subset (\T^k)^\Gamma$ obtained  by restricting to $X_A$ the shift action $\sigma$.
\par
In particular, if $A = f \in \Z[\Gamma]$, then \eqref{e:X-A-subshift} becomes
\begin{equation}
\label{e:X-f-subshift}
X_f= \{x \in \T^\Gamma : xf^*=0_{\T^\Gamma}\}.
\end{equation}
\par
Consider the surjective map $\pi \colon \ell^\infty(\Gamma)^k \to (\T^k)^\Gamma$
defined by $\pi(g)(\gamma)^i = g_\gamma^i \mod 1$ for all $g = (g^1, g^2, \ldots, g^k) \in \ell^\infty(\Gamma)^k$,
$\gamma \in \Gamma$, and $i=1,2, \ldots, k$.
Denote by $\ell^\infty(\Gamma,\Z)^k$ the set consisting of all $g \in \ell^\infty(\Gamma)^k$ such that $g_\gamma^i \in \Z$ for all $\gamma \in \Gamma$ and $i=1,2, \ldots, k$.

\begin{proposition}
\label{p:X-A-carac-lift}  Let $\Gamma$ be a countable group.
Let $x \in (\T^k)^\Gamma$ and $g \in \ell^\infty(\Gamma)^k$ such that $\pi(g) = x$.
With the above notation,  the following conditions are equivalent:
\begin{enumerate}[\rm (a)]
\item
$x \in X_A$;
\item
$g A^* \in \ell^\infty(\Gamma,\Z)^n$.
\end{enumerate}
\end{proposition}
\begin{proof}
This follows immediately from \eqref{e:X-A-subshift} and the equality $\pi(gA^*) = \pi(g)A^* = xA^*$.
\end{proof}

\subsection{The homoclinic group}
\label{sec:homoclinic}
Let $(X,\alpha)$ be an algebraic dynamical system with acting group $\Gamma$.
The set of points in $X$ that are homoclinic to $0_X$ with respect to $\alpha$ is
a $\Z[\Gamma]$-submodule  $\Delta(X,\alpha) \subset X$,  which is called the \emph{homoclinic group} of $(X,\alpha)$  (cf.~\cite{lind-schmidt}, \cite{lind-schmidt-survey-heisenberg}).
Note that $x \in \Delta(X,\alpha)$ if and only if
$\lim_{\gamma \to \infty} \gamma x = 0_X$.
We can choose a compatible metric $d$ on $X$ that is translation-invariant so that
$$
d(\gamma x, \gamma y) = d(\gamma x - \gamma y, 0_X) = d(\gamma(x - y),0_X)
$$
for all $x,y \in X$ and $\gamma \in \Gamma$.
We deduce that $x$ and $y$ are homoclinic if and only if $x - y \in \Delta(X,\alpha)$.

\begin{lemma}
\label{l:char-homoclinic-shift}
Let $\Gamma$ be a countable group, $k \in \N$, and let $x \in (\T^k)^\Gamma$.
The following conditions are equivalent.
\begin{enumerate}[{\rm (a)}]
\item
$x \in \Delta((\T^k)^\Gamma,\sigma)$;
\item
$\lim_{\gamma \to \infty} x(\gamma) = 0_{\T^k}$;
\item
there exists $g \in \CC_0(\Gamma)^k$ such that $x = \pi(g)$.
\end{enumerate}
\end{lemma}

\begin{proof}
Suppose (a). Then $\lim_{\gamma \to \infty} \gamma^{-1} x = 0_{(\T^k)^\Gamma}$.
As $x(\gamma) = \gamma^{-1}x(1_\Gamma)$, this implies
\[
\lim_{\gamma \to \infty} x(\gamma) = \lim_{\gamma \to \infty} \gamma^{-1}x(1_\Gamma) = 0_{\T^k}.
\]
This shows (a) $\implies$ (b).
\par
Conversely, suppose (b).
Let $W$ be a neighborhood of $0_{(\T^k)^\Gamma}$ in $(\T^k)^\Gamma$
and let us show that there exists a finite
subset $\Omega \subset \Gamma$ such that
\begin{equation}
\label{e:omega}
\gamma x \in W \mbox{ for all } \gamma \in \Gamma \setminus \Omega.
\end{equation}
By definition of the product topology, we can find a neighborhood $V$ of $0_{\T^k}$ in $\T^k$
and a finite subset $\Omega_1 \subset \Gamma$
such that
$W$ contains all $y \in (\T^k)^\Gamma$ that satisfy
\begin{equation*}
\label{e:W-V-Omega1}
y(\omega_1) \in V \mbox{ for all } \omega_1 \in \Omega_1.
\end{equation*}
On the other hand, since $\lim_{\gamma \to \infty} x(\gamma) = 0_{\T^k}$, we can find a finite subset
$\Omega_2 \subset \Gamma$
such that
\begin{equation}
\label{e:V}
x(\gamma) \in V \mbox{ for all } \gamma \in \Gamma  \setminus \Omega_2.
\end{equation}
Take  $\Omega \coloneqq \Omega_1 \Omega_2^{-1} \subset \Gamma$
and suppose that $\gamma \in \Gamma \setminus  \Omega$.
Then for every $\omega_1 \in \Omega_1$, we have that $\gamma^{-1} \omega_1 \in \Gamma \setminus \Omega_2$ and hence
\[
\gamma x(\omega_1) = x(\gamma^{-1} \omega_1)  \in V
\]
by~\eqref{e:V}.
This implies that  $\gamma x \in  W$.
Thus \eqref{e:omega} is satisfied.
This proves (b) $\implies$ (a).
\par
The fact that (c) implies (b) is an immediate consequence of the continuity of the quotient map
$\R^k \to \R^k/\Z^k = \T^k$.
Conversely, if we assume  (b), then the unique $g \in \ell^\infty(\Gamma)^k$ such that
$-1/2 \leq g_\gamma^j < 1/2$ and $x_j(\gamma) = g_\gamma^j \mod 1$ for all $\gamma \in \Gamma$ and $j=1,2,\ldots,k$,
clearly satisfies (c).
\end{proof}

\subsection{Connectedness of the phase space}
A non-zero element $f \in \Z[\Gamma]$ is called \emph{primitive} if there is no integer $n \geq 2$ that divides all coefficients of $f$. Every nonzero element $f \in \Z[\Gamma]$ can be uniquely written in the form
$f = m f_0$ with $m$ a positive integer and $f_0$ primitive. The integer $m$ is called the \emph{content} of $f$.
For principal algebraic dynamical systems with elementary amenable acting group we have the following criterion for connectedness
of the phase space.
Recall (cf.\ for instance \cite{Chou}) that the class of \emph{elementary amenable groups} is the smallest class of groups containing all finite groups and all Abelian groups that is closed under the operations of taking subgroups, quotiens, extensions, and direct limits.

\begin{proposition}
\label{p:Xf-connectedness}
Let $\Gamma$ be a countable torsion-free elementary amenable group $($e.g.\ $\Gamma = \Z^d)$.
Let $f \in \Z[\Gamma]$ with $f \not= 0$.
Then the following conditions are equivalent:
\begin{enumerate}[\rm (a)]
\item
$X_f$ is connected;
\item
$f$ is primitive.
\end{enumerate}
\end{proposition}
\begin{proof}
(a) $\Rightarrow$ (b):  Suppose that $f$ is not primitive. Then $f = m g$ for some integer $m \geq 2$ and $g \in \Z[\Gamma]$.
By \cite[Theorem 1.4]{KLM} the ring $\Q[\Gamma]$ is a domain. If $g=hf$ for some $h\in \Z[\Gamma]$, then $\frac{1}{m}f=hf$ in $\Q[\Gamma]$, and hence $h=\frac{1}{m} 1_\Gamma$, which is a contradiction. Thus $g+\Z[\Gamma] f$ is a nonzero element of
$\Z[\Gamma]/\Z[\Gamma] f$, while $m(g+\Z[\Gamma] f)=0$. Therefore $g+\Z[\Gamma] f$ is a nonzero torsion element of $M_f$, and hence $X_f$ is not connected.

(b) $\Rightarrow$ (a): Suppose that $f$ is primitive and that $X_f$ is not connected. Then $M_f$ has torsion, so that there exists
$a \in M_f$ with finite order $n\geq 2$. Replacing $a$ by some integral multiple of $a$, we may assume that $n$ is a prime number $p$.
Write $a=g+\Z[\Gamma] f$ for some $g\in \Z[\Gamma]$. Then $pg=hf$ for some $h\in \Z[\Gamma]$.
Denote by $\psi$ the natural ring morphism $\Z[\Gamma]\rightarrow (\Z/p\Z)[\Gamma]$ obtained by reducing coefficients modulo $p$.
Then $\psi(h)\psi(f)=0$. By \cite[Theorem~1.4]{KLM}, the ring $ (\Z/p\Z)[\Gamma]$ is a domain. Since $f$ is primitive, $\psi(f)\neq 0$.
Thus $\psi(h)=0$, i.e. $h=pw$ for some $w\in \Z[\Gamma]$. Then $pg=hf=pwf$, and hence $g=wf\in \Z[\Gamma]$.
This means that $a=0$, which is a contradiction.
\end{proof}

\section{Weak forms of expansivity for algebraic actions}
\label{s:weak-forms}
In this section we present and study weak forms of expansivity for algebraic actions.
This applies in particular to the harmonic models introduced in \cite{schmidt-verbitskiy} (see also \cite{Bowen-Li}).

\subsection{$p$-expansive algebraic actions and $p$-homoclinic groups}
\label{s:p-exp-and-p-homoclinic}
In this section we review the notions of $p$-expansive algebraic actions and of $p$-homoclinic groups
introduced by Chung and the third named author in \cite[Sections 4 and 5]{Chung-Li}.

Let $\Gamma$ be a countable group acting by automorphisms of a compact metrizable Abelian group $X$. Let also $1 \leq p \leq \infty$.
\par
For $x \in X$ and $\chi\in \widehat{X}$, define the function $\Psi'_{x, \chi}$ on  $\Gamma$ by setting
\begin{equation}
\label{e:PSI-prime}
\Psi'_{x, \chi}(\gamma)=e^{2\pi i\langle\gamma x, \chi\rangle}-1,
\end{equation}
for all  $\gamma\in \Gamma$.

One then says that the algebraic dynamical system $(X,\alpha)$  is \emph{$p$-expansive} provided there
exists a finite subset $W \subset \widehat{X}$ and $\varepsilon > 0$ such that $0_X$ is the only point $x \in X$ satisfying
\begin{equation}
\label{e:PSI-prime-epsilon}
\sum_{\chi \in W} \|\Psi'_{x, \chi}\|_p < \varepsilon.
\end{equation}

The following collects the main properties of $p$-expansivity.

\begin{theorem}[{\cite[Proposition 4.3 and Theorem 4.11]{Chung-Li}}]
\label{t:p-exp}
Let $(X,\Gamma,\alpha)$ be an algebraic dynamical system. Let $1 \leq p \leq \infty$. Then the following hold:
\begin{enumerate}[{\rm (1)}]
\item If $\alpha$ is $p$-expansive, then it is $q$-expansive for all $1 \leq q \leq p$.
\item If $\alpha$ is $p$-expansive, then $\widehat{X}$ is a finitely generated left $\Z[\Gamma]$-module.
\item If $\alpha$ is $p$-expansive, then for any finite subset $W \subset \widehat{X}$ generating
$\widehat{X}$ as a left $\Z[\Gamma]$-module, there exists $\varepsilon > 0$ such that $0_X$ is the only point $x \in X$ satisfying
\eqref{e:PSI-prime-epsilon}.
\item $\alpha$ is $\infty$-expansive if and only if it is expansive.
\item
\label{t:p-exp-last}
Let $k,n \in \N$ and $A \in \M_{n,k}(\Z[\Gamma])$. Then $\alpha_A$ is $p$-expansive if an only if the $\R[\Gamma]$-morphism
$\ell^p(\Gamma)^k \to \ell^p(\Gamma)^n$ sending $g$ to $gA^*$ is injective.
Moreover, if in addition, $\Gamma$ is amenable, then the following conditions are equivalent:
\begin{enumerate}[{\rm (a)}]
\item the topological entropy of $(X_A, \alpha_A)$ is finite;
\item $\alpha_A$ is $1$-expansive;
\item $\alpha_A$ is $2$-expansive;
\item the $\Z[\Gamma]$-morphism $\Z[\Gamma]^k \to \Z[\Gamma]^n$ sending $g$ to $gA^*$ is injective;
\item the $\R[\Gamma]$-morphism $\R[\Gamma]^k \to \R[\Gamma]^n$ sending $g$ to $gA^*$ is injective.
\end{enumerate}
\end{enumerate}
\end{theorem}

We now recall  the definitions of a $p$-homoclinic point and of the $p$-homoclinic group (cf.\ \cite[Section 5]{Chung-Li}).
Let $\Gamma$ be a countable group acting by automorphisms of the compact metrizable Abelian group $X$ and let $1 \leq p < \infty$.
One says that a point $x \in X$ is \emph{$p$-homoclinic} provided that $\Psi'_{x,\chi} \in \ell^p(\Gamma)$ for all $\chi
\in \widehat{X}$. Let then $\Delta^p(X,\alpha)$ denote the set of all $p$-homoclinic points of $X$. This is called
the \emph{$p$-homoclinic group} (cf.\ Theorem \ref{t:p-homo-group}.(2)) of the algebraic dynamical system $(X,\alpha)$.
Also one sets $\Delta^\infty(X,\alpha) \coloneqq \Delta(X,\alpha)$.
Note that for $p=1$, the set $\Delta^1(X,\alpha)$ was studied in \cite{schmidt-verbitskiy} and \cite{lind-schmidt-verbitskiy}.
Here below we collect some basic properties of the $p$-homoclinic groups.

\begin{theorem}[{\cite[Proposition 5.2, Lemma 5.3, and Lemma 5.4]{Chung-Li}}]
\label{t:p-homo-group}
Let $\Gamma$ be a countable group acting by automorphisms of the compact metrizable Abelian group $X$ and let $1 \leq p \leq \infty$.
Then the following hold:
\begin{enumerate}[{\rm (1)}]
\item One has $\Delta^p(X,\alpha) \subset \Delta^q(X,\alpha)$ for all $p \leq q \leq \infty$.
\item $\Delta^p(X,\alpha)$ is a $\Gamma$-invariant subgroup of $X$.
\item If $\alpha$ is $p$-expansive, then $\Delta^p(X,\alpha)$ is countable.
\item If $\Z[\Gamma]$ is left Noetherian and $\alpha$ is $p$-expansive, then $\Delta^p(X,\alpha)$ is a finitely generated
left $\Z[\Gamma]$-module.
\item
\label{item:sub}
Assume that $p < \infty$ and let $k,n \in \N$ and $A \in \M_{n,k}(\Z[\Gamma])$.
If $\alpha_A$ is $p$-expansive, then $\Delta^p(X_A,\alpha_A)$ is isomorphic to a $\Z[\Gamma]$-submodule of
$\Z[\Gamma]^n/\Z[\Gamma]^kA^*$. If, in addition, the  $\R[\Gamma]$-morphism
$\ell^p(\Gamma)^k \to \ell^p(\Gamma)^n$ sending $g$ to $gA^*$ is invertible,
then $\Delta^p(X_A,\alpha_A)$ is isomorphic to $\Z[\Gamma]^n/\Z[\Gamma]^kA^*$.
\end{enumerate}
\end{theorem}

\subsection{Homoclinically expansive actions}
In this section we introduce and study a new form of weak expansivity for dynamical systems.
\begin{definition}
\label{def:h-e}
Let $(X,\alpha)$ be a dynamical system with acting group $\Gamma$.
One says that the action is \emph{homoclinically expansive} if there exists a constant $\varepsilon_0 > 0$ such
that, for each pair of distinct homoclinic points $x,y \in X$, there exists an element $\gamma \in \Gamma$ such that
$d(\gamma x, \gamma y) > \varepsilon_0$, where $d$ is any compatible metric on $X$.
Such a constant $\varepsilon_0$ is then called a \emph{homoclinic-expansivity constant} for $(X,\alpha,d)$.
\end{definition}

Note that the fact that $(X, \alpha)$ is homoclinically expansive is in fact independent of the choice of the metric $d$
by compactness of $X$. A pseudometric $\tilde{d}$ on $X$ is said to be \emph{dynamically-generating}
if for all distinct $x, y \in X$ there is $\gamma \in \Gamma$ such that $\tilde{d}(\gamma x, \gamma y) > 0$ (cf.\ \cite[Definition 9.35]{kerr-li}). Now, given a dynamically-generating continuous pseudometric $\tilde{d}$ on $X$, we can define a compatible metric $d$ on $X$ by
setting
$$d(x, y)=\sum_{n=0}^\infty\frac{1}{2^n}\tilde{d}(\gamma_nx, \gamma_ny),$$
where $\gamma_0=1_\Gamma, \gamma_1, \ldots$ is an enumeration of the elements of $\Gamma$. Then
$$\sup_{\gamma\in \Gamma}\tilde{d}(\gamma x, \gamma y) \le \sup_{\gamma \in \Gamma}d(\gamma x, \gamma y)\le 2\sup_{\gamma\in \Gamma}
\tilde{d}(\gamma x, \gamma y)$$
for all $x, y\in X$. Thus in Definition \ref{def:h-e} we may take $d$ to be any dynamically-generating continuous pseudometric on $X$.

In the following, we study homoclinic expansivity for algebraic actions.
Let $(X,\alpha)$ be an algebraic dynamical system with acting group $\Gamma$.

For any $t\in \R$, we set $|t+\Z| \coloneqq \min_{m\in \Z} |t+m|$.
More generally, for $k \in \N$ and $t = (t_1, t_2, \ldots, t_k)\in \R^k$ we set
\begin{equation}
\label{e:dist-k}
|t+\Z^k| \coloneqq \max_{1 \leq j \leq k} |t_j+\Z|.
\end{equation}

Given $x\in X$ and $\chi\in \widehat{X}$, define a function $\Psi_{x, \chi}$ on $\Gamma$ by setting
\begin{equation}
\label{e:PSI}
\Psi_{x, \chi}(\gamma)=|\langle\gamma x, \chi\rangle|,
\end{equation}
for all  $\gamma\in \Gamma$.
\par
As a comparison between \eqref{e:PSI-prime} and \eqref{e:PSI},
note that there is some constant $C>0$ such that
\[
C|t|\le |e^{2\pi i t}-1|\le C^{-1}|t|
\]
for all $t\in [-1/2, 1/2]$. It follows that, for all $1\le p\le \infty$,
\[
C\|\Psi_{x, \chi}\|_p \le \|\Psi'_{x, \chi}\|_p \le C^{-1}\|\Psi_{x, \chi}\|_p.
\]

It is easy to see that for any $x\in X$, one has that $x\in \Delta(X, \alpha)$ if and only if $\Psi_{x, \chi}\in \CC_0(\Gamma)$ for all
$\chi\in \widehat{X}$.

\begin{proposition}
\label{P-h expansive}
Let $(X,\alpha)$ be an algebraic dynamical system with acting group $\Gamma$. Then the following conditions are equivalent:
\begin{enumerate}[{\rm (a)}]
\item the action $\alpha$ is homoclinically expansive;
\item there exist a finite subset $W$ of $\widehat{X}$ and $\varepsilon>0$ such that $0_X$ is the only point $x$ in $\Delta(X, \alpha)$ satisfying
\begin{equation}
\label{e:W-hE}
\sum_{\chi \in W}\|\Psi_{x, \chi}\|_\infty<\varepsilon.
\end{equation}
\end{enumerate}
\end{proposition}
\begin{proof}
Take a function $h\in \ell^1(\widehat{X})$ such that $h(\chi)>0$ for every $\chi\in \widehat{X}$.
Consider the compatible metric $\rho$ on $X$ defined by
\[
\rho(x, y)=\sum_{\chi\in \widehat{X}} h(\chi)|\langle x, \chi\rangle - \langle y, \chi\rangle|=\sum_{\chi\in \widehat{X}} h(\chi)|\langle x-y, \chi\rangle|.
\]

Assume that (b) holds with $W$ and $\varepsilon$ and let us set $\varepsilon_0 \coloneqq \varepsilon \min_{\chi\in W}h(\chi)/|W|$.
Let $(x, y)$ be a homoclinic pair of $X$ with $\sup_{\gamma \in \Gamma}\rho(\gamma x, \gamma y)<\varepsilon_0$. Note that
\begin{align*}
\sup_{\gamma \in \Gamma}\rho(\gamma x, \gamma y)&=\sup_{\gamma \in \Gamma}\sum_{\chi\in \widehat{X}} h(\chi)|\langle\gamma x-\gamma y, \chi\rangle|\\
&=\sup_{\gamma \in \Gamma}\sum_{\chi\in \widehat{X}} h(\chi)\Psi_{x-y, \chi}(\gamma)\\
&\ge \sup_{\gamma\in \Gamma} (\min_{\chi\in W}h(\chi))\sum_{\chi\in W}\Psi_{x-y, \chi}(\gamma)\\
&\ge \frac{\min_{\chi\in W}h(\chi)}{|W|}\sum_{\chi \in W}\|\Psi_{x-y, \chi}\|_\infty.
\end{align*}
Thus $\sum_{\chi \in W}\|\Psi_{x-y, \chi}\|_\infty<\varepsilon$, and hence $x=y$. This shows that $\varepsilon_0$ is a
homoclinic-expansivity constant for $(X, \alpha)$, and (b)$\Rightarrow$(a) follows.

Now assume that (a) holds and let $\varepsilon_0 > 0$ be a homoclinic-expansivity constant
for $(X, \alpha)$. Take a finite subset $W$ of $\widehat{X}$ such that $\sum_{\chi\in \widehat{X}\setminus W}h(\chi)<\varepsilon_0/2$
and set $\varepsilon \coloneqq \varepsilon_0/(2\max_{\chi\in W}h(\chi))$.
Let $x\in \Delta(X,\alpha)$ with $\sum_{\chi \in W}\|\Psi_{x, \chi}\|_\infty<\varepsilon$.
Then
\begin{align*}
\sup_{\gamma\in \Gamma}\rho(\gamma x, \gamma 0_X)&=\sup_{\gamma \in \Gamma}\sum_{\chi\in \widehat{X}}h(\chi)\Psi_{x, \chi}(\gamma)\\
&\le \varepsilon_0/2+\sup_{\gamma \in \Gamma}\sum_{\chi\in W}h(\chi)\Psi_{x, \chi}(\gamma)\\
&\le \varepsilon_0/2+(\max_{\chi\in W}h(\chi))\sum_{\chi\in W}\|\Psi_{x, \chi}\|_\infty\\
& < \varepsilon_0,
\end{align*}
and hence $x=0_X$. This shows that (a)$\Rightarrow$(b).
\end{proof}

\begin{proposition}
\label{P-h expansive basic}
Let $(X,\Gamma,\alpha)$ be an algebraic dynamical system. Suppose that $\alpha$ is homoclinically expansive and $\widehat{X}$ is finitely generated (as a left $\Z[\Gamma]$-module).
Then the following hold:
\begin{enumerate}[{\rm (1)}]
\item $\alpha$ is $p$-expansive for all $1\le p<\infty$;
\item for any finite subset $W$ of $\widehat{X}$ generating $\widehat{X}$ as a left $\Z[\Gamma]$-module, there exists $\varepsilon>0$ such that $0_X$ is the only point $x$ in $\Delta(X, \alpha)$ satisfying $\sum_{\chi \in W}\|\Psi_{x, \chi}\|_\infty<\varepsilon$.
\end{enumerate}
\end{proposition}
\begin{proof} Note that for any $x\in X$, $\chi, \chi'\in \widehat{X}$, and $u,v\in \Z[\Gamma]$ we have
\begin{align} \label{E-norm}
\|\Psi_{x, u\chi+v\chi'}\|_q\le \|u\|_1\|\Psi_{x, \chi}\|_q+\|v\|_1\|\Psi_{x, \chi'}\|_q
\end{align}
for all $1\le q\le \infty$.

(1). Since $\alpha$ is homoclinically expansive, by virtue of Proposition~\ref{P-h expansive} we can find a finite subset $W$ of
$\widehat{X}$ and $\varepsilon>0$ satisfying \eqref{e:W-hE}. Enlarging $W$ if necessary, we may assume that $W$ generates $\widehat{X}$
as a left $\Z[\Gamma]$-module.
Let $1\le p<\infty$. If $x\in X$ and $\sum_{\chi\in W}\|\Psi_{x, \chi}\|_p<\varepsilon$, then from \eqref{E-norm} we know that
$x\in \Delta^p(X, \alpha)\subset \Delta(X,\alpha)$, and using $\sum_{\chi\in W}\|\Psi_{x, \chi}\|_\infty\le \sum_{\chi\in W}\|\Psi_{x, \chi}\|_p<\varepsilon$, we conclude that $x=0_X$. Therefore $\alpha$ is $p$-expansive.

(2) follows from Proposition~\ref{P-h expansive} and \eqref{E-norm}.
\end{proof}

For finitely presented (e.g.\ principal) algebraic actions we have the following characterization of homoclinic expansivity
(compare with Theorem \ref{t:p-exp}.(\ref{t:p-exp-last})).

\begin{theorem}
\label{t:char-homo-exp}
Let $\Gamma$ be a countable group. Let $k, n\in \N$ and $A\in \M_{n,k}(\Z[\Gamma])$.
Then the following conditions are equivalent:
\begin{enumerate}[{\rm (a)}]
\item $(X_A,\alpha_A)$ is homoclinically expansive;
\item the linear map $\CC_0(\Gamma)^k\rightarrow \CC_0(\Gamma)^n$ sending $g$ to $gA^*$ is injective.
\end{enumerate}
\end{theorem}

\begin{proof}
For $x \in X_A$ consider the function $\Phi_x \in \ell^\infty(\Gamma)$
defined by $\Phi_x(\gamma) \coloneqq |x_\gamma|$ for all $\gamma \in \Gamma$, where $| \cdot |$ is as in
\eqref{e:dist-k}.

Assume that (b) fails. Then $gA^* = 0$ for some nonzero $g\in \CC_0(\Gamma)^k$.
As a consequence, for all $\lambda\in \R$ one also has $\lambda gA^* = 0$ and hence $\pi(\lambda g) \in \Delta(X_A, \alpha_A)$,
by Proposition \ref{p:X-A-carac-lift} and Lemma \ref{l:char-homoclinic-shift}.
Let $W \subset \widehat{X_A}$ denote the image of the canonical basis of $\Z[\Gamma]^k$ under the quotient map
$\Z[\Gamma]^k \to \Z[\Gamma]^k/\Z[\Gamma]^nA$. Note that $W$ then generates $\widehat{X_A}$ as a left $\Z[\Gamma]$-module.
When $\lambda \to 0$, one has $\|\Phi_{\pi(\lambda g)}\|_\infty \to 0$, and hence $\sum_{\chi \in W}\|\Psi_{\pi(\lambda g), \chi}\|_\infty \to 0$. Since $g \neq 0$, when $|\lambda|$ is sufficiently small and nonzero, $\pi(\lambda g) \neq 0_{X_A}$.
From Proposition \ref{P-h expansive basic}.(2), we deduce that $(X_A,\alpha_A)$ is not homoclinically expansive.
This shows (a)$\Rightarrow$(b).
\par
Now assume that (b) holds so that, in particular, $A\neq 0$. Let $d$ be a translation-invariant compatible metric on $X_A$.
Then there is some $\varepsilon_0 > 0$ such that for any $x \in X_A$ with  $d(x,0_{X_A})\le \varepsilon_0$, one has $|x_{1_\Gamma}|<1/(2\|A\|_1)$.
Let $x, y \in X_A$ be two homoclinic points with $\max_{\gamma \in \Gamma} d(\gamma x, \gamma y)\le
\varepsilon_0$. Then $x-y \in \Delta(X_A, \alpha_A)$, and for any $\gamma \in \Gamma$ we have $d(\gamma (x-y), 0_{X_A})\le \varepsilon_0$, and hence $|(x-y)_{\gamma^{-1}}|=|(\gamma(x-y))_{1_\Gamma}|<1/(2\|A\|_1)$.
Let $g$ be the unique element of  $\ell^\infty(\Gamma)^k$ satisfying $\|g\|_\infty\le1/(2\|A\|_1)$ and $\pi(g)=x-y$.
Since $x-y$ is in $\Delta(X_A, \alpha_A)$, we have $g\in \CC_0(\Gamma)^k$.
It follows that $\|gA^*\|_\infty\le \|g\|_\infty\|A\|_1\le 1/2$ and, by Proposition \ref{p:X-A-carac-lift},
$gA^*\in \ell^\infty(\Gamma, \Z)$. Therefore, $gA^*=0$. By (b) we have $g= 0$, and hence $x=y$.  This shows that
$\varepsilon_0$ is a homoclinic expansivity constant for $(X_A, \alpha_A)$, and (b)$\Rightarrow$(a) follows as well.
\end{proof}

Note that when $n=k=1$ and $A = f \in \Z[\Gamma]$,
condition (b) in Theorem \ref{t:char-homo-exp} (and therefore homoclinic expansivity of $\alpha_f$) is equivalent
to (we-1) in Definition \ref{def:we}.

\begin{proposition}
\label{p:he-implies}
Let $\Gamma$ be a countable group. Let $k, n\in \N$ and $A\in \M_{n,k}(\Z[\Gamma])$ and suppose that
$(X_A,\alpha_A)$ is homoclinically expansive. Then
$\Delta(X_A, \alpha_A)$ is isomorphic to a left $\Z[\Gamma]$-submodule of $\Z[\Gamma]^n/\Z[\Gamma]^k A^*$.
\end{proposition}
\begin{proof}
For each $x\in \Delta(X_A, \alpha_A)$, take $\tilde{x}\in \CC_0(\Gamma)^k$ with $\pi(\tilde{x})=x$
(cf.\ Lemma \ref{l:char-homoclinic-shift}). Then, by virtue of Proposition \ref{p:X-A-carac-lift},
$\tilde{x}A^*\in \ell^\infty(\Gamma, \Z)^n\cap \CC_0(\Gamma)^n=\Z[\Gamma]^n$.
If we choose another $\tilde{x}'\in \CC_0(\Gamma)^k$ with $\pi(\tilde{x}')=x$, then $\tilde{x}-\tilde{x}'\in \ell^\infty(\Gamma, \Z)^k\cap \CC_0(\Gamma)^k=\Z[\Gamma]^k$, and hence $\tilde{x} A^*-\tilde{x}'A^*\in \Z[\Gamma]^k A^*$.
Thus the map $\varphi \colon \Delta(X_A, \alpha_A)\rightarrow \Z[\Gamma]^n/\Z[\Gamma]^k A^*$ sending $x$ to $\tilde{x} A^*+\Z[\Gamma]^k A^*$ is well defined. Clearly, $\varphi$ is a left $\Z[\Gamma]$-module morphism. Let $x\in \ker(\varphi)$. Then $\tilde{x} A^*=g A^*$,
for some $g\in \Z[\Gamma]^k$. By virtue of Theorem \ref{t:char-homo-exp}, we have $\tilde{x}=g$ and hence $x=0_{X_A}$.
Thus $\varphi$ is injective.
\end{proof}

\begin{examples}
\label{e:examples-HE}
Here below we describe some examples of principal algebraic actions and discuss their homoclinic expansivivity.
\begin{enumerate}[{\rm (1)}]
\item Suppose that $\gamma \in \Gamma$ has infinite order and denote by $\Gamma' \cong \Z$ the subgroup of $\Gamma$ it generates.
It follows from \cite[Theorem 5.1]{linnel:analytic} that for any nonzero $f \in \C[\Gamma']$ and any nonzero $g \in \CC_0(\Gamma')$, one
has $fg \neq 0$. Using the right-coset decomposition of $\Gamma$, it follows that for any nonzero $f \in \Z[\Gamma']$ and any nonzero
$g \in \CC_0(\Gamma)$, one has $fg \neq 0$. This shows that the associated principal algebraic action $\alpha_f$ is homoclinically expansive.
Note that if $f \coloneqq 1_\Gamma - \gamma \in \Z[\Gamma']$, then $f$ is not invertible in $\ell^1(\Gamma)$ and hence,
by \cite[Theorem 3.2]{deninger-schmidt} (cf.\ Theorem \ref{t:pads-expansive} below), $\alpha_f$ is not expansive.
(cf.\ \cite[Example 4.6]{Chung-Li}.)

\item Let $\Gamma = \Z^d$ and let $f \in \R[\Gamma]$. Note that $f$ may be regarded, by duality, as a function
on $\widehat{\Gamma} = \T^d$. We denote by $Z(f) \subset \T^d$ its zero set. It follows from
\cite[Theorem 2.2]{linnel-puls} that $Z(f)$ is contained in the image of the intersection of $[0,1]^d$ and a finite union of hyperplanes in
$\R^d$ under the natural quotient map $P \colon \R^d \to \T^d$ if and only if $fg \neq 0$ for all nonzero $g \in \CC_0(\Gamma)$.
From Theorem \ref{t:char-homo-exp} we thus deduce that the principal algebraic action associated with $f \in \Z[\Gamma]$ is homoclinically expansive if and only if $Z(f)$ is contained in the image of the intersection of $[0,1]^d$ and a finite union of hyperplanes in $\R^d$ under $P$. This is the case, for instance, if $Z(f)$ is finite. (cf.\ \cite[Example 4.9]{Chung-Li}.)

\item Suppose that $\Gamma$ contains two elements $\gamma, \gamma' \in \Gamma$ that generate a non-Abelian free subsemigroup.
Consider the polynomial
$f \coloneqq \pm3 \cdot 1_\Gamma - (1_\Gamma +\gamma-\gamma^2)\gamma' \in \Z[\Gamma]$. It follows from an argument similar to that in
\cite[Example 7.2]{lind-schmidt} that the associated principal algebraic action is homoclinically expansive
(though not expansive by \cite[Example A.1]{Li:annals} and \cite[Theorem 3.2]{deninger-schmidt}). (cf.\ \cite[Example 4.10]{Chung-Li}.)

\item In \cite[Example 4.7]{Chung-Li} it is shown that for $\Gamma = \Z^d$, $d \geq 2$, the element $h = 2d-1 - \sum_{j=1}^d (u_j+u_j^{-1}) \in \Z[u_1,u_1^{-1}, \cdots, u_d, u_d^{-1}] = \Z[\Z^d]$ satisfies that, for any $\frac{2d}{d-1} < p \leq +\infty$, the corresponding principal
algebraic action $\alpha_h$ is not $p$-expansive. It follows from Proposition \ref{P-h expansive basic}.(1) that $\alpha_h$ is not
homoclinically expansive either.
\end{enumerate}
\end{examples}

\subsection{Principal algebraic actions associated with weakly expansive polynomials}
\begin{theorem}
\label{t:we-implies}
Let $\Gamma$ be a countable group and suppose that $f \in \Z[\Gamma]$ is weakly expansive.
Then the following hold:
\begin{enumerate}[{\rm (1)}]
\item The element $\omega\in \CC_0(\Gamma)$ satisfying {\rm (we-2)} in Definition \ref{def:we} is unique and, moreover, $\omega f = 1_\Gamma$.

\item $\Delta(X_f, \alpha_f)$ is dense in $X_f$.

\item $\Delta(X_f, \alpha_f)$ is isomorphic to $\Z[\Gamma]/\Z[\Gamma]f^*$ as a left $\Z[\Gamma]$-module.
\end{enumerate}
\end{theorem}
\begin{proof}
Let $\omega, \omega' \in \CC_0(\Gamma)$ satisfying (we-2). Then $f\omega = 1_\Gamma = f\omega'$ yields
$f(\omega - \omega') = 0$ and condition (we-1) in Definition \ref{def:we} infers $\omega = \omega'$. This proves uniqueness of $\omega$.
Moreover, from Lemma \ref{l:associative} and (we-2) we deduce
\[
f(\omega f) = (f \omega) f = 1_\Gamma f = f = f 1_\Gamma.
\]
Thus, $f(\omega f - 1_\Gamma) = 0$, and, again by (we-1), we get $\omega f = 1_\Gamma$. This shows (1).

In order to prove (2), let now $a\in \Z[\Gamma]/\Z[\Gamma] f$ with $\langle a, \pi(\gamma\omega^*)\rangle=0$ for all $\gamma\in \Gamma$
(where $\pi \colon \ell^\infty(\Gamma)\rightarrow \T^\Gamma$ is as in Subsection \ref{ss:PAAds}).
%%%
Write $a=g+\Z[\Gamma] f$ for some $g\in \Z[\Gamma]$. Then
\[
\begin{split}
0 & = \langle a, \pi(\gamma\omega^*)\rangle\\
& = \sum_{\delta \in \Gamma} (\gamma\omega^*)_\delta g_\delta + \Z \\
& = \sum_{\delta \in \Gamma} (\gamma\omega^*)_\delta (g^*)_{\delta^{-1}} + \Z\\
& = (\gamma\omega^* g^*)_{1_\Gamma}+\Z\\
& =(\omega^* g^*)_{\gamma^{-1}}+\Z\\
& = (g\omega)_\gamma + \Z
\end{split}
\]
for all $\gamma\in \Gamma$. Thus $h\coloneqq g \omega$ lies in $\ell^\infty(\Gamma, \Z)$. Since $\omega\in \CC_0(\Gamma)$ and $g\in \Z[
\Gamma]$, one has $g\omega\in \CC_0(\Gamma)$. Therefore $h\in \ell^\infty(\Gamma, \Z)\cap \CC_0(\Gamma)=\Z[\Gamma]$. Using
Lemma \ref{l:associative} and (1) it follows that
\[
h f= (g \omega)f = g(\omega f) = g 1_\Gamma = g,
\]
and hence $g=hf\in \Z[\Gamma] f$, which means that $a=0$.
By Lemma \ref{l:char-homoclinic-shift} we have $\{\pi(\gamma\omega^*): \gamma \in \Gamma\} \subset \Delta(X_f, \alpha_f)$ and by Pontryagin duality we conclude that $\Delta(X_f, \alpha_f)$ is dense in $X_f$.

We are only left to prove (3). In the proof of Proposition \eqref{p:he-implies} (here we take $n=k=1$ and $A = f$)
we have defined an injective left $\Z[\Gamma]$-module morphism $\varphi \colon \Delta(X_f, \alpha_f)\rightarrow \Z[\Gamma]/\Z[\Gamma] f^*$ sending $x$ to $\tilde{x}f^*+\Z[\Gamma] f^*$. Let us show that $\varphi$ is surjective. Let $h\in \Z[\Gamma]$.
Using (we-2) and Lemma \ref{l:associative} we deduce that
\begin{multline*}
\varphi(h\pi(\omega^*))= \varphi(\pi(h \omega^*)) = (h \omega^*)f^* + \Z[\Gamma] f^* =
h(\omega^* f^*) + \Z[\Gamma] f^*\\ = h(f\omega)^* + \Z[\Gamma] f^* =
h 1_\Gamma + \Z[\Gamma] f^* = h+\Z[\Gamma] f^*.
\end{multline*}
This shows that $\varphi$ is surjective. Therefore $\varphi$ is indeed an isomorphism.
\end{proof}

It follows from the proof of Theorem \ref{t:we-implies}.(3) and the notation therein that the cyclic $\Z[\Gamma]$-module $\Delta(X_f,\alpha_f)$ is generated by the element $x^\Delta \coloneqq \pi(\omega^*) \in \Delta(X_f,\alpha_f)$, called the \emph{fundamental homoclinic point}
of $(X_f,\alpha_f)$ (cf.\ \cite{lind-schmidt}).

Let $(X,\alpha)$ be an algebraic dynamical system with an infinite acting group $\Gamma$ and denote by $\mu$ the Haar probability measure on $X$. Let $r \in \N$ with $r \geq 2$. One says that $(X,\alpha)$ is \emph{mixing of order $r$} if, for all measurable subsets $B_1, B_2,
\ldots, B_r \subset X$, one has
\begin{equation}
\label{e:def-mixing}
\mu(\gamma_1 B_{1} \cap \gamma_2 B_{2} \cap \cdots \cap \gamma_r B_{r})
 \longrightarrow \mu(B_{1})\mu(B_{2})\cdots \mu(B_{r})
\end{equation}
as $\gamma_i^{-1}\gamma_j \to \infty$ in $\Gamma$ for all $1 \leq i <j \leq r$. If $(X,\alpha)$ is mixing of order $r=2$, one simply
says that $(X,\alpha)$ is \emph{mixing}. Note that every mixing algebraic dynamical system is ergodic.
If $(X,\alpha)$ is mixing of order $r$ for all $r \geq 2$, one then says that $(X,\alpha)$ is \emph{mixing of all orders}.
\par
Observe that if $A$ is a compact metrizable Abelian group and $\Gamma$ is any infinite countable group, then the $\Gamma$-shift $(A^\Gamma,\sigma)$ is mixing of all orders since~\eqref{e:def-mixing} is trivially satisfied when the $B_i$s are cylinders, for all $r \geq 2$.
\par
It follows from \cite[Proposition 4.6]{Bowen-Li} that an algebraic dynamical system $(X,\alpha)$ admitting a dense homoclinic group
is mixing of all orders. From Theorem \ref{t:we-implies}.(2) we then deduce the following:

\begin{corollary}
\label{c:mixing}
Let $\Gamma$ be an infinite countable group and suppose that $f \in \Z[\Gamma]$ is weakly expansive. Then the associated
algebraic dynamical system $(X_f,\alpha_f)$ is mixing of all orders.
\end{corollary}

\subsection{Expansive principal algebraic actions}
\label{ss:expansivity}
The following result is due to Deninger and Schmidt~\cite[Theorem~3.2]{deninger-schmidt}
(see also \cite[Theorem~5.1]{lind-schmidt-survey-heisenberg}).

\begin{theorem}
\label{t:pads-expansive}
Let $\Gamma$ be a countable group and $f \in \Z[\Gamma]$.
Then the following conditions are equivalent:
\begin{enumerate}[\rm (a)]
\item
the dynamical system $(X_f,\alpha_f)$ is expansive;
\item
$f$ is invertible in $\ell^1(\Gamma)$.
\end{enumerate}
\end{theorem}

For other characterizations of expansivity for algebraic dynamical systems we refer to \cite{schmidt, schmidt-book}
(for $\Gamma = \Z^d$, $d \in \N$), \cite{miles} (for $\Gamma$ Abelian), \cite{einsiedler-rindler} (for
$(X,\alpha)$ finitely presented),  \cite{bhattacharya2} (for $X$ connected and finite-dimensional),
and \cite[Theorem 3.1]{Chung-Li}.

As observed in~\cite{lind-schmidt-survey-heisenberg}, if $f$ is \emph{lopsided}, i.e., there exists an element
$\gamma_0 \in \Gamma$ such that
$|f_{\gamma_0}| > \sum_{\gamma \not= \gamma_0} |f_\gamma|$,
then $f$ is invertible in $\ell^1(\Gamma)$.
On the other hand, there are $f \in \Z[\Gamma]$ invertible in $\ell^1(\Gamma)$ that are not lopsided.
For instance, if we take $\Gamma = \Z$, then the polynomial
$u^2 - u - 1 \in \Z[\Gamma] = \Z[u,u^{-1}]$
 is not lopsided although it is invertible in $\ell^1(\Gamma)$
(the associated principal algebraic dynamical system  is
conjugate to the $\Z$-system generated by Arnold's cat map $(x_1,x_2) \mapsto (x_2,x_1 + x_2)$ on the $2$-dimensional torus $\T^2$, see
e.g.~\cite[Example~2.18.(2)]{schmidt-book}).

The following result justifies our terminology for weakly expansive polynomials.

\begin{corollary}
\label{c:espansive-vs-weakly expansive}
Let $\Gamma$ be a countable group and $f \in \Z[\Gamma]$. Suppose that the dynamical system $(X_f,\alpha_f)$ is expansive.
Then $f$ is weakly expansive.
\end{corollary}
\begin{proof}
Expansivity of $(X_f, \alpha_f)$ implies, by Theorem \ref{t:pads-expansive}, that $f$ is invertible in $\ell^1(\Gamma)$.
Then $\omega \coloneqq f^{-1} \in \ell^1(\Gamma) \subset \CC_0(\Gamma)$ yields (we-2) in Definition \ref{def:we}.
\par
Let now $g \in \CC_0(\Gamma)$ and suppose that $fg = 0$. Then, recalling \eqref{e:associativity}, we deduce that
\[
0 = \omega 0 = \omega (fg) = (\omega f) g = 1_\Gamma g = g,
\]
and (we-1) follows as well.
\end{proof}

\subsection{Harmonic models}
\label{sec:harmonic}
Let $f \in \Z[\Gamma]$ be well-balanced. It follows from \eqref{e:X-f-subshift} that $x \in \T^\Gamma$ belongs to $X_f$
if and only if $x$ satisfies the {\it harmonicity equation} (mod $1$)
\[
\sum_{\eta \in \Gamma} f_{\eta} x(\gamma \eta)  = 0,
\]
for all $\gamma \in \Gamma$. This explains the terminology. Note that for $\Gamma = \Z^d$, the polynomial $f \in \Z[\Gamma] = \Z[u_1, u_1^{-1}, \ldots, u_d, u_d^{-1}]$ given by
\[
f = 2d - \sum_{i=1}^d (u_i + u_i^{-1})
\]
is well-balanced and the corresponding harmonicity equation is the discrete analogue of the Laplace equation (cf.\ Eq. (4.5) in
\cite{schmidt-verbitskiy}).

\begin{lemma}
\label{l:f-star-injective}
Let $\Gamma$ be a countable infinite group and $f \in \R[\Gamma]$. Suppose that $f$ is well-balanced.
Then the map $g \mapsto fg$ from $\CC_0(\Gamma)$ to $\CC_0(\Gamma)$ is injective.
In particular, harmonic models are homoclinically expansive.
\end{lemma}
\begin{proof} Set
\begin{equation}
\label{e:mu}
\mu \coloneqq 1_\Gamma - \frac{1}{f_{1_\Gamma}}f  \in \R[\Gamma].
\end{equation}
Then $\mu$ is a probability measure on $\Gamma$ which is symmetric and its support $S \coloneqq \supp(\mu)$ generates $\Gamma$ as a semigroup, by ({\rm wb-1}), ({\rm wb-3}), and ({\rm wb-4}), respectively.
\par
In order to show (we-1) we apply the maximum principle. Let $g \in \CC_0(\Gamma)$ and suppose that $fg = 0$, equivalently, $\mu g = g$.
Set $M \coloneqq \max_{\delta \in \Gamma} |g_\delta|$ and observe that $A \coloneqq \{\gamma \in \Gamma: |g_\gamma| = M\}$ is non-empty. Moreover, if $\gamma \in A$ one has, using the triangle inequality and the properties of $\mu$ we alluded to above,
\[
M = |g_\gamma| = |(\mu g)_\gamma| \leq \sum_{\delta \in S} \mu_{\delta^{-1}} |g_{\delta\gamma}| \leq  \sum_{\delta \in S} M \mu_\delta = M,
\]
forcing $|g_{\delta \gamma}| = M$ for all $\delta \in S$. This shows that $SA \subset A$. A recursive argument immediately shows that
$S^nA \subset A$ for all $n \in \N$. Since $S$ generates $\Gamma$ as a semigroup, we get that $A = \Gamma$. In other words, $|g|$ is a constant function. As $g \in \CC_0(\Gamma)$, we conclude that $g = 0$.
\par
The last statement follows immediately after Theorem \ref{t:char-homo-exp}.
\end{proof}

In the arguments preceding Lemma 4.8 in \cite{Bowen-Li} it is shown that if $\Gamma$ is a countable infinite group which is not virtually $\Z$ or $\Z^2$ and $f \in \Z[\Gamma]$ is well-balanced, then
\begin{equation}
\label{e:omega3}
\omega \coloneqq \frac{1}{f_{1_\Gamma}} \sum_{k = 0}^\infty \mu^k \in \CC_0(\Gamma),
\end{equation}
where $\mu$ is as in \eqref{e:mu}, satisfies that $f\omega = 1_\Gamma$, so that (we-2) holds.
Combining this with Lemma \ref{l:f-star-injective}, we deduce:

\begin{proposition}
\label{p:wb-implies-we}
Let $\Gamma$ be a countable infinite group $\Gamma$ that is not virtually $\Z$ or $\Z^2$.
Then every balanced polynomial $f \in \Z[\Gamma]$ is weakly expansive.
\end{proposition}

\begin{remark}
It is a well known fact in the theory of Markov chains (cf.\ for instance \cite[Definition 1.14]{woess}) that the sum in \eqref{e:omega3} expressing $\omega$ is (pointwise) convergent if and only if the random walk on $\Gamma$ associated with $\mu$  is \emph{transient} (i.e., given any finite subset $\Omega \subset \Gamma$, there exists $t(\Omega) \in \N$ such that, with probability one, the position $x(t) \in \Gamma$ of the random walker on $\Gamma$ at time $t \geq t(\Omega)$ satisfies that $x(t) \in \Gamma \setminus \Omega$) and it is a deep result of Varopoulos (cf.\ \cite{varopoulos, varopoulos-s-c} and \cite[Theorem 3.24]{woess}) that this is the case exactly if $\Gamma$ is not virtually $\Z$ or $\Z^2$.
\end{remark}

% SECTION 4
\section{Topological rigidity}
\label{sec:affine}

\subsection{Affine maps}
Let $X$ and $Y$ be two topological Abelian groups.
A map $\tau \colon Y \to X$ is called \emph{affine} if there is a continuous group morphism
$\lambda \colon Y \to X$ and an element $t \in X$ such that
$\tau(y) = \lambda(y) + t$ for all $y \in Y$.
Note that $\lambda$ and  $t$ are then uniquely determined by $\tau$
since they must satisfy $t = \lambda(0_Y)$ and $\lambda(y) = \tau(y) - t$ for all $y \in Y$.
One says that
$\lambda$ and $t$ are respectively the \emph{linear part}  and the \emph{translational part} of the affine map $\tau$.
\par
The following two obvious criteria will be useful in the sequel.

\begin{proposition}
\label{p:pre-inj-affine}
Let $(X,\alpha)$ be an algebraic dynamical system and let $\tau \colon X \to X$ be an affine map
with linear part $\lambda \colon X \to X$.
Then the following conditions are equivalent:
\begin{enumerate}[\rm (a)]
\item
$\tau$ is pre-injective;
\item
$\lambda$ is pre-injective;
\item
$\Ker(\lambda) \cap \Delta(X,\alpha) = \{0_X\}$.\end{enumerate}
\end{proposition}

\begin{proposition}
\label{p:charact-equiv-affine}
Let $(X,\alpha)$ be an algebraic dynamical system and let $\tau \colon X \to X$ be an affine map
with linear part $\lambda \colon X \to X$ and translational part $t \in X$.
Then the following conditions are equivalent:
\begin{enumerate}[\rm (a)]
\item
$\tau$ is $\Gamma$-equivariant;
\item
$\lambda$ is $\Gamma$-equivariant and $t \in \Fix(X,\alpha)$.
\end{enumerate}
\end{proposition}

\subsection{Topological rigidity}
One says that the algebraic dynamical system $(X,\alpha)$ is \emph{topologically rigid} if every
endomorphism $\tau \colon X \to X$ of $(X,\alpha)$ is affine.

Before stating our rigidity results, let us introduce some notation.
Let $L(X)$ denote the real vector space of all group homomorphisms $\widehat{X} \to \R$ equipped with the
topology of pointwise convergence. Note that $\Gamma$ acts on $L(X)$ by setting $[\gamma \psi](\chi) \coloneqq
\psi(\gamma^{-1}\chi)$ for all $\psi \in L(X)$ and $\chi \in \widehat{X}$. Moreover, the map $E \colon
L(X) \to X$ defined by $E(\psi)(\chi) \coloneqq \psi(\chi) + \Z$ for all $\psi \in L(X)$ and
$\chi \in \widehat{X}$ is a continuous ($\Gamma$-equivariant) group homomorphism.

\begin{theorem}
\label{T-rigidity}
Let $(X, \alpha)$ and $(Y, \beta)$ be algebraic dynamical systems with acting group $\Gamma$. Suppose that $Y$ is connected and that the homoclinic group $\Delta(Y, \beta)$ is dense in $Y$. Also suppose that $(X, \alpha)$ is homoclinically expansive.
Then every $\Gamma$-equivariant continuous map $Y\rightarrow X$ is affine.
\end{theorem}
\begin{proof} Let $\tau \colon Y \to X$ be a $\Gamma$-equivariant continuous map. By Bhattacharya's extension of van Kampen theorem
\cite[Theorem 1]{bhattacharya}, there are a $\Gamma$-equivariant affine map $\lambda \colon Y\to X$ and a $\Gamma$-equivariant continuous map $\Phi \colon Y \to L(X)$ such that $\Phi(0_Y)=0$ and $\tau=\lambda + E\circ \Phi$.
(Note that in the statement of \cite[Theorem 1]{bhattacharya}, $X$ is assumed to be connected as well; however, in its proof, this condition
is never used.)
Thus it suffices to show that $\Phi=0$.
Let $y\in \Delta(Y, \beta)$ and $\chi \in \widehat{X}$. For any $\gamma\in \Gamma$ we have
\begin{align*}
[\Phi(\gamma y)](\chi)=[\gamma \Phi(y)](\chi)=[\Phi(y)](\gamma^{-1}\chi).
\end{align*}
When $\gamma\to \infty$, we have $\Phi(\gamma y)\to \Phi(0_Y)=0$, and hence $[\Phi(y)](\gamma^{-1}\chi)\to 0$.
For any $t\in \R$, we have
\begin{multline*}
%\label{e:affine-ineq}
\Psi_{E(t\Phi(y)), \chi}(\gamma)=|\langle\gamma E(t\Phi(y)), \chi\rangle|=|\langle E(t\Phi(y)), \gamma^{-1}\chi\rangle|\\
\le |[t\Phi(y)](\gamma^{-1}\chi)|=|t|\cdot|[\Phi(y)](\gamma^{-1}\chi)|
\end{multline*}
for all $\gamma\in \Gamma$, and hence $\Psi_{E(t\Phi(y)), \chi}\in \CC_0(\Gamma)$. Therefore $E(t\Psi(y))\in \Delta(X, \alpha)$.
Since $\alpha$ is homoclinically expansive, by Proposition~\ref{P-h expansive} there exist a finite subset $W$ of $\widehat{X}$ and
$\varepsilon>0$ such that $0_X$ is the only point $x$ in $\Delta(X, \alpha)$ satisfying
\[
\sum_{\chi \in W}\|\Psi_{x, \chi}\|_\infty<\varepsilon.
\]
Set $C \coloneqq \sum_{\chi\in W}\sup_{\gamma\in \Gamma}|[\Phi(y)](\gamma^{-1}\chi)|<\infty$. Then
\[
\sum_{\chi\in W}\|\Psi_{E(t\Phi(y)), \chi}\|_\infty\le |t|C.
\]
Thus for all $t\in \R$ with $|t|<\varepsilon/C$ we have $E(t\Phi(y))=0_X$, which means that $t\Phi(y)$ takes integer values.
It follows that $\Phi(y)=0$. Since $\Delta(Y, \beta)$ is dense in $Y$ and $\Phi$ is continuous, we conclude that $\Phi=0$ as desired.
\end{proof}

\begin{corollary}
\label{c:endo-princ}
Let $\Gamma$ be a countable group. Suppose that $f \in \Z[\Gamma]$ satisfies that $(X_f,\alpha_f)$ is homoclinically expansive
and $\Delta(X_f,\alpha_f)$ is dense in $X_f$ (e.g., that $f$ is weakly expansive), and the phase space $X_f$ is connected.
Then $(X_f, \alpha_f)$ is topologically rigid.
\par
If, in addition, $\Gamma$ is Abelian, then for a map $\tau \colon X_f \to X_f$ the following conditions are equivalent:
\begin{enumerate}[\rm(a)]
\item
$\tau$ is an endomorphism of the dynamical system $(X_f,\alpha_f)$;
\item
there exist $r \in \Z[\Gamma]$ and $t \in \Fix(X_f,\alpha_f)$ such that
$\tau(x) = r x + t$ for all $x \in X_f$.
\end{enumerate}
\end{corollary}

\begin{proof} The first statement follows immediately from Theorem \ref{T-rigidity} after taking
$X = Y = X_f$ (if $f$ is weakly expansive, recall Theorem \ref{t:char-homo-exp} and Theorem \ref{t:we-implies}.(2)).
\par
Suppose now that $\Gamma$ is Abelian and let $\tau \colon X_f \to X_f$ be a map.
For each $r \in \Z[\Gamma]$,  the self-mapping of $X_f$ given by $x \mapsto r x$ is
$\Gamma$-equivariant since $\Z[\Gamma]$ is commutative.
Therefore, the fact that (b) implies (a) follows from Proposition~\ref{p:charact-equiv-affine}.
\par
To prove the converse, suppose that $\tau$ is an endomorphism of the dynamical
system $(X_f,\alpha_f)$.
It follows from the first statement that $\tau$ is affine.
Therefore, by using Proposition~\ref{p:charact-equiv-affine}, there exist a continuous $\Z[\Gamma]$-module  morphism
$\lambda \colon X_f \to X_f$ and $t \in \Fix(X_f,\alpha_f)$ such that
$\tau(x) = \lambda(x) + t$ for all $x \in X_f$.
As the ring $\Z[\Gamma]$ is commutative and
$\widehat{X_f} = \Z[\Gamma]/ \Z[\Gamma] f$ is a cyclic $\Z[\Gamma]$-module,
there is $s \in \Z[\Gamma]$ such that $\widehat{\lambda}(\chi) = s \chi$ for all $\chi \in \widehat{X_f}$.
Since $\lambda = \widehat{\widehat{\lambda}}$, setting $r \coloneqq s^* \in \Z[\Gamma]$ it follows that $\lambda(x) = s^* x = r x$,
and hence $\tau(x) = r x + t$, for all $x \in X_f$.
\end{proof}

% SECTION 5
\section{Proof of Theorem~\ref{t:main-result}}
\label{sec:proof}
%In this section, we present the proof of Theorem~\ref{t:main-result}.

\begin{proof}[Proof of Theorem~\ref{t:main-result}]
Let $\Gamma$ be a countable Abelian group and let $f \in \Z[\Gamma]$ be a weakly expansive polynomial such that
$X_f$ is connected. Also let $\tau$ be an endomorphism of $(X_f,\alpha_f)$, i.e.,
a $\Gamma$-equivariant continuous map $\tau \colon X_f \to X_f$.
We want to show that $\tau$ is surjective if and only if it is pre-injective.
 \par
By Corollary~\ref{c:endo-princ}, there exists $r \in \Z[\Gamma]$ such that
$\tau$ is affine with linear part $\lambda \colon X_f \to X_f$ given by
$\lambda(x) = r x$ for all $x \in X_f$.
Clearly $\tau$ is surjective if and only if $\lambda$ is.
As $X_f$ is compact, we know that surjectivity of $\lambda$ is equivalent to injectivity of its Pontryagin dual
$\widehat{\lambda} \colon \widehat{X_f} \to \widehat{X_f}$.
Now we observe that $\widehat{\lambda}(\chi) = r^* \chi$ for all $\chi \in \widehat{X_f}$.
\par
As $\Gamma$ is Abelian, the natural mapping
$\Phi \colon \widehat{X_f}=\Z[\Gamma]/\Z[\Gamma]f \to \Z[\Gamma]/\Z[\Gamma]f^*$ defined by
\[
\Phi(g + \Z[\Gamma]f) = g^* + \Z[\Gamma]f^*
\]
for all $g \in \Z[\Gamma]$ is a group isomorphism.
Note that
\[
\Phi(g_1(g_2 + \Z[\Gamma]f))  = \Phi(g_1g_2 + \Z[\Gamma]f) = (g_1g_2)^* + \Z[\Gamma]f^* =  g_1^*g_2^* +
\Z[\Gamma]f^* = g_1^*\Phi(g_2 + \Z[\Gamma]f)
\]
for all $g_1,g_2 \in \Z[\Gamma]$.
As a consequence, denoting by $\kappa \colon \widehat{X_f}  \to  \Delta(X_f, \alpha_f)$ the composition of $\Phi$ with the $\Z[\Gamma]$-module isomorphism $\Z[\Gamma]/\Z[\Gamma]f^* \to \Delta(X_f,\alpha_f)$ in Theorem \ref{t:we-implies}.(3),
we have the commuting diagram
\[
\begin{CD}
\widehat{X_f} @> {\chi\mapsto r^*\chi}>> \widehat{X_f}\\
@V{\kappa}VV  @VV{\kappa}V \\
\Delta(X_f, \alpha_f)    @> {x\mapsto r x}>> \Delta(X_f, \alpha_f).
\end{CD}
\]
\vspace{0.5cm}

We deduce that injectivity of $\widehat{\lambda}$ is equivalent to injectivity of the
group endomorphism $\mu$ of $\Delta(X_f,\alpha_f)$ defined by $\mu(x) \coloneqq r x$ for all $x \in \Delta(X_f,\alpha_f)$.
As $\mu$ is the restriction of $\lambda$ to $\Delta(X_f,\alpha_f)$, we conclude that surjectivity of $\tau$ is equivalent
to pre-injectivity of $\tau$, by using Proposition~\ref{p:pre-inj-affine}.
\end{proof}

From the proof of Theorem \ref{t:main-result} and Proposition \ref{p:he-implies} we immediately deduce the following:

\begin{corollary}
\label{c:homo-exp-Moore}
Let $\Gamma$ be a countable Abelian group and let $f\in \Z[\Gamma]$. Suppose that
$(X_f, \alpha_f)$ is homoclinically expansive, $\Delta(X_f, \alpha_f)$ is dense in $X_f$, and that $X_f$ is connected.
Then $(X_f, \alpha_f)$ has the Moore property.
\end{corollary}

\section{Atoral polynomials and proof of Theorem \ref{t:GOE-irr-finite-zero-set}}
\label{s:atoral}
Lind, Schmidt, and Verbistkiy \cite[Theorem 3.2]{lind-schmidt-verbitskiy-2} gave the following geometric-dynamical characterization of atorality for irreducible polynomials in $\Z[\Z^d]$.
\begin{theorem}
\label{t:LSV}
Let $\Gamma = \Z^d$ and suppose that $f \in \Z[\Gamma]$ is irreducible. Then
the following conditions are equivalent:
\begin{enumerate}[{\rm (a)}]
\item $\Delta^1(X_f, \alpha_f)\neq \{0_{X_f}\}$;
\item $\Delta^1(X_f, \alpha_f)$ is dense in $X_f$;
\item $f$ is \emph{atoral} in the sense that there is some $g\in \Z[\Gamma]$ such that $g\not \in \Z[\Gamma] f$ and $Z(f)\subset Z(g)$;
\item $\dim Z(f)\le d-2$.
\end{enumerate}
\end{theorem}

The meaning of $\dim(\cdot)$ in Theorem \ref{t:LSV}.(d) is explained in \cite[page 1063]{lind-schmidt-verbitskiy-2}; in particular,
one has $\dim(\varnothing) \coloneqq - \infty$.
Also remark that, if $d = 1$, an irreducible polynomial $f \in \Z[\Z] = \Z[u_1,u_1^{-1}]$ is atoral if and only if $Z(f) = \varnothing$
and this, in turn, is equivalent to $(X_f,\alpha_f)$ being expansive (cf.\ \cite[Lemmma 2.1.(1)]{lind-schmidt}).

\begin{examples}
Here below, we present some examples of irreducible atoral polynomials $f \in \Z[\Z^d]$, mainly from  \cite[Section 3]{lind-schmidt-verbitskiy} and \cite[Section 4]{lind-schmidt-verbitskiy-2}. We can then apply Theorem \ref{t:GOE-irr-finite-zero-set} and deduce
that the corresponding algebraic dynamical systems $(X_f, \alpha_f)$ satisfy the Garden of Eden theorem.

\begin{enumerate}[{\rm (1)}]
\item Let $d=1$ and $f(u) = u^2 - u - 1 \in \Z[u,u^{-1}] = \Z[\Z]$. Then $f$ is irreducible and, since $Z(f) = \varnothing$, atoral.
The associated principal algebraic dynamical system $(X_f,\alpha_f)$ is conjugated to the hyperbolic dynamical system
$(\T^2,\varphi_A)$, where $A = \begin{pmatrix} 0 & 1\\1 & 1 \end{pmatrix} \in \GL(2,\Z)$ is \emph{Arnold's cat} and
$\varphi_A \colon \T^2 \to \T^2$ is the associated hyperbolic automorphism of the two-dimensional torus.
Note that $(X_f,\alpha_f)$ is expansive so that we can deduce the Moore-Myhill property also from Corollary \ref{c:expansive} (this is a particular case of the Garden of Eden theorem for Anosov diffeomorphisms on tori~\cite{csc-ijm-2015}, we alluded to in the Introduction).

\item Let $d = 2$ and $f(u_1,u_2) = 2 - u_1 - u_2 \in \Z[u_1,u_1^{-1},u_2,u_2^{-1}] = \Z[\Z^2]$. Then
$Z(f) = \{(1,1)\}$, and so $f$ is atoral.
Thus, by Example \ref{e:examples-HE}.(2), $(X_f, \alpha_f)$ is homoclinically expansive.
Also, it follows from \cite[\S 5]{lind-schmidt-verbitskiy} and \cite[Example 4.3]{lind-schmidt-verbitskiy-2})
that there is some $\omega$ in $\CC_0(\Z^2)$ with $f\omega=\omega f=1_\Gamma$.
As a consequence, $f$ is weakly expansive (though not well-balanced).
Moreover, $f$ is also primitive, so that, by Proposition \ref{p:Xf-connectedness}, $X_f$ is connected.
Applying Theorem \ref{t:main-result}, we obtain an alternative proof of the fact that $(X_f,\alpha_f)$ has the Moore-Myhill property.

\item Let $d=2$, and consider the harmonic model $f(u_1,u_2) = 4 - u_1-u_1^{-1}-u_2-u_2^{-1} \in \Z[u_1,u_1^{-1},u_2,u_2^{-1}] = \Z[\Z^2]$. One has $Z(f)=\{(1, 1)\}$. Thus $f$ is atoral and $(X_f, \alpha_f)$ satisfies the Garden of Eden theorem, by virtue of Theorem \ref{t:GOE-irr-finite-zero-set}. (Note that we cannot apply Theorem \ref{t:main-result}.)

\item Let $d=2$, and $f(u_1,u_2) = 1 + u_1 + u_2 \in \Z[u_1,u_1^{-1},u_2,u_2^{-1}] = \Z[\Z^2]$. Then $Z(f)=\{(\omega, \omega^2),
(\omega^2, \omega)\}$, where $\omega = \exp(2\pi i/3)$. The algebraic dynamical system $(X_f, \alpha_f)$ is called the \emph{connected Ledrappier shift}.

\item Let $d=2$, and $f(u_1,u_2) = 2 - u_1^2 + u_2 -u_1u_2 \in \Z[u_1,u_1^{-1},u_2,u_2^{-1}] = \Z[\Z^2]$. One has $Z(f)=\{(\xi,\eta),
(\overline{\xi},\overline{\eta})\}$, where $\xi, \eta$ are  algebraic numbers.

\item Let $d=2$, and $f(u_1,u_2) = 2 - u_1^3 + u_2 -u_1u_2 - u_1^2u_2 \in \Z[u_1,u_1^{-1},u_2,u_2^{-1}] = \Z[\Z^2]$. Here $Z(f)=\{(1,1), (i, \zeta), (-i, \overline{\zeta}), (\xi,\eta), (\overline{\xi},\overline{\eta})\}$, where $\zeta, \xi, \eta$ are algebraic numbers.

\item Let $d=3$ and $f(u_1,u_2,u_3) = 1 + u_1 + u_2 + u_3 \in \Z[u_1,u_1^{-1},u_2,u_2^{-1}, u_3,u_3^{-1}] = \Z[\Z^3]$. The corresponding zero-set $Z(f)$ is the union of three circles, namely, $\{(-1,s,-s): s \in \T\}$, $\{(s,-1,-s): s \in \T\}$,
and $\{(s,-s,-1): s \in \T\}$. Hence, $f$ is atoral.

\item Let $d=3$ and $f(u_1,u_2,u_3) = 3 + 3u_1 - 3 u_1^3 + u_1^4 - u_2 - u_3  \in \Z[u_1,u_1^{-1},u_2,u_2^{-1}, u_3,u_3^{-1}] = \Z[\Z^3]$.
One has has $Z(f)=\{(\eta, \overline{\eta}, \overline{\eta}), (\overline{\eta},\eta,\eta)\}$, where $\eta$ is an algebraic integer.
\end{enumerate}
\end{examples}

We are now in position to prove Theorem \ref{t:GOE-irr-finite-zero-set}.

\begin{proof}[Proof of Theorem \ref{t:GOE-irr-finite-zero-set}]
We first note that, by Example \ref{e:examples-HE}.(2), $\alpha_f$ is homoclinically expansive. Moreover, by Theorem \ref{t:LSV},
$\Delta^1(X_f,\alpha_f)$ and therefore $\Delta(X_f,\alpha_f)$ are dense in $X_f$. Finally, since $f$ is irreducible it is primitive
and hence (cf.\ Proposition \ref{p:Xf-connectedness}) $X_f$ is connected. Thus we may apply Corollary~\ref{c:endo-princ}
and deduce that $(X_f, \alpha_f)$ is topologically rigid.\\
\noindent
(a)$\Rightarrow$(b) is given by Corollary \ref{c:homo-exp-Moore}.\\
(b)$\Rightarrow$(c) follows from $\Delta^p(X_f, \alpha_f)\subset \Delta(X_f, \alpha_f)$ for all $1 \leq p \leq \infty$
(cf.\ Theorem \ref{t:p-homo-group}.(1)). \\
(c)$\Rightarrow$(d) this is trivial.\\
(d)$\Rightarrow$(a). Let $\tau \colon X_f \to X_f$ be a $\Gamma$-equivariant $1$-pre-injective continuous map.
By topological rigidity, there exists $r\in \Z[\Gamma]$ and $t \in X_f$ such that the group endomorphism $\lambda$ of $X_f$
defined by setting $\lambda(x) \coloneqq rx$ for all $x \in X_f$ satisfies that $\tau(x) = \lambda(x) + t$ for all $x \in X_f$.
Then the dual map $\widehat{\lambda} \colon \widehat{X_f} \to \widehat{X_f}$ is given by $\widehat{\lambda}(\chi) = r^*\chi$
for all $\chi \in \widehat{X_f}$.
Set $\textbf{m}_f=\Z[\Gamma]\cap (f \ell^1(\Gamma))$, which is an ideal of $\Z[\Gamma]$. For each $g\in \textbf{m}_f$, one has
$g=fv_g $ for a unique $v_g\in \ell^1(\Gamma)$. Then $x^{(g)} \coloneqq \pi(v_g^*)$ lies in $\Delta^1(X_f, \alpha_f)$, where $\pi$ is, as usual,  the projection map $\ell^\infty(\Gamma)\to \T^\Gamma$.
The map $g+\Z[\Gamma]f \mapsto x^{(g)}$ is clearly a group isomorphism from $\textbf{m}_f/\Z[\Gamma]f$ onto $\Delta^1(X_f, \alpha_f)$
(cf.\ \cite[Corollary 3.3]{lind-schmidt-verbitskiy-2}).
Take $g\in \textbf{m}_f\setminus \Z[\Gamma]f$. Since $f$ is irreducible, the group homomorphism $h+\Z[\Gamma]f\mapsto hg+\Z[\Gamma]f$
from $\widehat{X_f}$ to $\textbf{m}_f/\Z[\Gamma]f$ is injective. Thus the group homomorphism
$\overline{\kappa} \colon \widehat{X_f} \to \Delta^1(X_f, \alpha_f)$ defined by $\overline{\kappa}(h+\Z[\Gamma]f) \coloneqq x^{(hg)}$ is injective.
We have the commutative diagram:
\[
\begin{CD}
\widehat{X_f} @> {\chi\mapsto r^*\chi} >>  \widehat{X_f} \\
@V{\overline{\kappa}}VV  @VV{\overline{\kappa}}V \\
\Delta^1(X_f, \alpha_f)  @> {x\mapsto r x}>> \Delta^1(X_f, \alpha_f).
\end{CD}
\]

Since $\tau$ is $1$-pre-injective, so is $\lambda$, that is, $\lambda$ is
injective on $\Delta^1(X_f, \alpha_f)$. Hence the dual map $\widehat{\lambda}$ is injective. By Pontryaging duality, this is equivalent
to $\lambda$ being surjective. It follows that $\tau$ is surjective as well.
\end{proof}

% SECTION 7
\section{Concluding remarks}

\subsection{Surjunctivity}
A dynamical system $(X,\alpha)$ is called \emph{surjunctive} if every injective endomorphism of
$(X,\alpha)$ is surjective (and hence a homeomorphism).
As injectivity implies pre-injectivity, we deduce from Theorem~\ref{t:main-result}
that if $\Gamma$ is a countable Abelian group, $f \in \Z[\Gamma]$ is weakly expansive
and $X_f$ is connected, then the dynamical system  $(X_f,\alpha_f)$ is surjunctive.
Actually, in the case $\Gamma = \Z^d$, this last result is a particular case
of Theorem~1.5 in~\cite{bcsc-surjunctivity}
which asserts that if $\Gamma = \Z^d$ and $M$ is a finitely generated $\Z[\Gamma]$-module, then
$(\widehat{M},\alpha_M)$ is surjunctive.

\subsection{Counterexamples for mixing algebraic dynamical systems}
Recall (cf.\ Corollary \ref{c:mixing}) that if $\Gamma$ is infinite and $f \in \Z[\Gamma]$ is weakly expansive,
then the associated algebraic dynamical system $(X_f,\alpha_f)$ is mixing of all orders.
\par
The examples below show that Theorem~\ref{t:main-result} becomes false if the hypothesis that $f \in \Z[\Gamma]$ is weakly expansive is replaced by the weaker hypotheses that the homoclinic group $\Delta(X_f, \alpha_f)$ is dense and that $(X_f, \alpha_f)$ is mixing.

\begin{example}
Let $\T = \R/\Z$, $\Gamma = \Z^d$, $d \geq 1$,
and consider the $\Gamma$-shift $(\T^\Gamma,\sigma)$
(this is  $(X_f,\alpha_f)$ for $f = 0 \in \Z[\Gamma]$).
Then the endomorphism $\tau$ of $(\T^\Gamma,\sigma)$ defined by
$\tau(x)(\gamma) = 2x(\gamma)$ for all $x \in \T^\Gamma$ and $\gamma \in \Gamma$,
is clearly surjective.
However, $\tau$ is not pre-injective since the configuration $y \in \T^\Gamma$, defined by
$y(\gamma) = 1/2 \mod 1$ if $\gamma = 1_\Gamma$ and $0$ otherwise, is a non-trivial element in the homoclinic group
of $(\T^\Gamma,\sigma)$ and satisfies
$\tau(y) = \tau(0_{\T^\Gamma}) = 0_{\T^\Gamma}$.
It follows that $(\T^\Gamma,\sigma)$ does not have the Moore property.
\end{example}

\begin{example}
Let $\Gamma = \Z$ and consider the  polynomial
\[
f = 1 -2u_1 + u_1^2 - 2u_1^3 + u_1^4 \in  \Z[u_1, u_1^{-1}] = \Z[\Gamma].
\]
The associated algebraic dynamical system $(X_f,\alpha_f)$ is conjugate
to the system $(\T^4,\beta)$,
where $\beta$ is the action of $\Z$ on $\T^4$ generated by the companion matrix of $f$.
It is mixing since $f$ is not divisible by a cyclotomic polynomial
(cf.~\cite[Theorem~6.5.(2)]{schmidt-book}).
On the other hand,  $f$  has four distinct roots in $\C$,
two   on the unit circle, one inside and one outside.
As $f$ is irreducible over $\Q$, it follows that   
  the homoclinic group $\Delta(X_f,\alpha_f)$  is reduced to $0$
  (cf.~\cite[Example 3.4]{lind-schmidt}).
The trivial endomorphism of $(X_f,\alpha_f)$, that maps each $x \in X_f$ to $0$,
is pre-injective but not surjective.
Consequently, $(X_f,\alpha_f)$ does not have the Myhill property.
However,  $(X_f,\alpha_f)$ has the Moore property
since  each homoclinicity class of $(X_f,\alpha_f)$ is reduced to a single point,
so that every map with source set $X_f$  is pre-injective.
Note  that $(X_f,\alpha_f) = (\T^4,\beta)$ is topologically rigid since
every mixing toral automorphism is topologically rigid by
a result of Walters~\cite{walters}.
\end{example}

\subsection{$p$-pre-injectivity and the $p$-Moore and $p$-Myhill properties}
In Subsection \ref{s:p-exp-and-p-homoclinic} we have recalled from \cite{Chung-Li} the notions and the main properties
of $p$-expansivity and of $p$-homoclinic group (denoted $\Delta^p(X,\alpha)$) for an algebraic dynamical system $(X,\alpha)$.
In the Introduction we have also defined the notion of a $p$-pre-injective map $\tau \colon X \to X$.

Note that, if $\tau$ is a group endomorphism of $X$, then (cf.\ Proposition \ref{p:pre-inj-affine}):
\begin{equation}
\label{e:p-pre-inj-endo}
\mbox{$\tau$ is $p$-pre-injective if and only if
$\ker(\tau) \cap \Delta^p(X,\alpha) = \{0_X\}$.}
\end{equation}
\par
We shall then say that the algebraic dynamical system $(X,\alpha)$ has the \emph{$p$-Moore property}
if every surjective endomorphism of $(X,\alpha)$ is $p$-pre-injective and that it has the \emph{$p$-Myhill property} if
every $p$-pre-injective endomorphism of $(X,\alpha)$ is surjective. Note that the $\infty$-Moore property (resp.\
$\infty$-Myhill property) is nothing but the Moore property (resp.\ Myhill property).
\par
It follows immediately from Theorem \ref{t:p-homo-group}.(1) that if $1 \leq p \leq q \leq \infty$ the every
$q$-pre-injective map $\tau \colon X \to X$ is $p$-pre-injective, so that
every algebraic dynamical system satisfying the $q$-Moore property (resp.\ the $p$-Myhill property) satisfies the
$p$-Moore property (resp.\ the $q$-Myhill property).

From the proof of Theorem \ref{t:main-result} and from Theorem \ref{t:p-homo-group}.(\ref{item:sub})
we immediately deduce the following:

\begin{corollary}
\label{c:p-Moore}
Let $\Gamma$ be a countable Abelian group and let $f\in \Z[\Gamma]$. Let also $1\le p<\infty$.
Suppose that $(X_f, \alpha_f)$ is $p$-expansive and topologically rigid.
Then $(X_f, \alpha_f)$ has the $p$-Moore property.
\end{corollary}

In the following result we relax the commutativity condition for the acting group $\Gamma$.

\begin{theorem}
Let $(X,\alpha)$ be a finitely generated algebraic dynamical system.
Suppose that $\Gamma$ is amenable and that $(X, \alpha)$ is topologically rigid
and has finite entropy. Then $(X, \alpha)$ has the $1$-Moore property.
\end{theorem}
\begin{proof}
Let $\tau \colon X \to X$ be a continuous $\Gamma$-equivariant surjective map.
By topological rigidity, we can find a continuous group endomorphism $\lambda \colon X \to X$
and $t \in X$ such that $\tau(x) = \lambda(x) + t$ for all $x \in X$. Since $\tau$ is
surjective (resp.\ injective) if and only if $\lambda$ is surjective (resp.\ injective), it
is not restrictive to suppose that $t = 0$, that is, $\tau$ is a group endomorphism of $X$.
Then the \emph{entropy addition formula} (cf.\ \cite[Corollary 6.3]{Li:annals}, see also \cite[Theorem 13.48]{kerr-li}) yields
(we denote by $h(\cdot)$ topological entropy)
\[
h(\ker(\tau)) = h(X) - h(X/\ker(\tau)) = h(X) - h(\tau(X)) = 0,
\]
since $\tau$ is surjective and $X$ has finite entropy. For any finitely generated algebraic dynamical system $(Y, \beta)$ with acting group $\Gamma$, if $\Delta^1(Y, \beta)$ is nontrivial, then $h(Y)>0$ (cf.\  \cite[Theorem 7.3 and Propositions 5.7 and 7.10]{Chung-Li}, see also \cite[Theorems 13.23 and 13.35]{kerr-li}). Since $\widehat{X}$ is a finitely generated left $\Z[\Gamma]$-module and $\widehat{\ker(\tau)}$ is  a quotient  $\Z[\Gamma]$-module  of  $\widehat{X}$, $\widehat{\ker(\tau)}$ is also a finitely generated left $\Z[\Gamma]$-module.
It follows that 
$$\Delta^1(X, \alpha)\cap \ker(\tau)=\Delta^1(\ker(\tau), \alpha|_{\ker(\tau)})=\{0_X\}.$$ 
This, together with \eqref{e:p-pre-inj-endo}, shows that $\tau$ is $1$-pre-injective.
\end{proof}

 \bibliographystyle{siam}

\end{document}